\documentclass[
11pt,
letterpaper,
onecolumn,
oneside]{article}

\usepackage[margin=1 in,letterpaper]{geometry}
\usepackage[singlespacing]{setspace}
\usepackage{fancyhdr}
\usepackage[inline]{enumitem}
\usepackage{newclude}
\usepackage{caption}
\usepackage{subcaption}
\usepackage[utf8]{inputenc}
\usepackage[T1]{fontenc}
\usepackage{lmodern}
\usepackage{appendix}
\usepackage{comment}
\usepackage{graphicx}
\graphicspath{ {images/} }
\usepackage[colorlinks=true,citecolor=black,linkcolor=black,urlcolor=black]{hyperref} 
\usepackage{amssymb}

\usepackage{amsthm}
\usepackage{mathtools}
\usepackage{bm}
\usepackage{mathrsfs}
\usepackage{bbm}
\usepackage{xcolor}

\theoremstyle{definition}
\newtheorem{definition}{Definition}[section]
\newtheorem{remark}[definition]{Remark}

\newtheorem{question}[definition]{Question}

\theoremstyle{plain}
\newtheorem{theorem}[definition]{Theorem}
\newtheorem{proposition}[definition]{Proposition}
\newtheorem{lemma}[definition]{Lemma}
\newtheorem{corollary}[definition]{Corollary}

\newcommand*{\mrmc}{\mathrm{c}}
\newcommand*{\mrme}{\mathrm{e}}

\newcommand*{\mcle}{\mathcal E}
\newcommand*{\mclh}{\mathcal H}

\newcommand*{\expe}{\mathbb E}

\newcommand*{\setle}{\subseteq}

\newcommand{\HitH}{{T_{H}}}
\newcommand{\HitG}{{T_{G}}}
\newcommand{\HitHS}{{T_{H}^*}}
\newcommand{\HitGS}{{T_{G}^*}}
\newcommand{\HitHP}{{T'_{H}}}

\newcommand{\histj}{{\mathbf{Hist}_j}}

\newcommand{\Bone}{{\mathcal{B}_1}}
\newcommand{\Btwo}{{\mathcal{B}_2}}
\newcommand{\Bthree}{{\mathcal{B}_3}}
\newcommand{\Bfour}{{\mathcal{B}_4}}
\newcommand{\Bfive}{{\mathcal{B}_5}}

\newcommand{\bN}{{\mathbf{N}}}
\newcommand{\bY}{{\mathbf{Y}}}
\newcommand{\bS}{{\mathbf{S}}}

\DeclareMathOperator{\cl}{cl}

\renewcommand{\epsilon}{\varepsilon} 

\newcommand{\mc}{\mathcal}

\renewcommand{\le}{\leqslant}
\renewcommand{\leq}{\leqslant}
\renewcommand{\ge}{\geqslant}
\renewcommand{\geq}{\leqslant}

\begin{document}
	\title{The hitting time of clique factors }
	\author{Annika Heckel\thanks{Matematiska institutionen, Uppsala universitet, Box 480, 751 06 Uppsala, Sweden. Email: \href{mailto:annika.heckel@math.uu.se}{\nolinkurl{annika.heckel@math.uu.se}}. The research leading to these results has received funding from the European Research Council, ERC Grant Agreement 772606–PTRCSP, and from the Swedish Research Council, reg.~nr. 2022-02829.}
		\and
		Marc Kaufmann\thanks{Institut für Theoretische Informatik, ETH Zürich, Zürich, Switzerland. Email: \href{mailto:marc.kaufmann@inf.ethz.ch}{\nolinkurl{marc.kaufmann@inf.ethz.ch}}. The author gratefully acknowledges support by the Swiss National Science Foundation [grant number 200021\_192079].}
		\and
		Noela Müller\thanks{Department of Mathematics and Computer Science, Eindhoven University of Technology, PO Box 513, 5600 MB Eindhoven, The Netherlands. Email: \href{mailto:n.s.muller@tue.nl}{\nolinkurl{n.s.muller@tue.nl}}. Research supported by NWO Gravitation project NETWORKS under grant no. 024.002.003}
		\and
		Matija Pasch\thanks{Mathematisches Institut, Ludwig-Maximilians-Universität München, Theresienstr. 39, 80333 München, Germany. Email: \href{mailto:pasch@math.lmu.de} {\nolinkurl{pasch@math.lmu.de}}. The research leading to these results has received funding from the European Research Council, ERC Grant Agreement 772606–PTRCSP.}
	}
 \date{February 16, 2023}
	\maketitle

 \begin{abstract}
  In \cite{kahn2022hitting}, Kahn gave the strongest possible, affirmative, answer to Shamir's problem, which had been open since the late 1970s: Let $r \ge 3 $ and let $n$ be divisible by $r$. Then, in the random $r$-uniform hypergraph process on $n$ vertices, as soon as the last isolated vertex disappears, a perfect matching emerges. In the present work, we transfer this hitting time result to the setting of clique factors in the random graph process: At the time that the last vertex joins a copy of the complete graph $K_r$, the random graph process contains a $K_r$-factor. Our proof draws on a novel sequence of couplings, extending techniques of Riordan \cite{riordan2022random} and the first author \cite{heckel2020random}. An analogous result is proved for clique factors in the  $s$-uniform hypergraph process ($s \ge 3$). 
     
 \end{abstract}

	\section{Introduction}
When can we cover the vertices of a graph with disjoint isomorphic copies of a small subgraph? The study of this question goes back at least to 1891, when Julius Petersen, in his \textit{Theorie der regulären graphs} \cite{petersen1891}, provided sufficient conditions for a graph to contain a perfect matching, that is, a vertex cover with pairwise disjoint edges. In the realm of random (hyper-)graphs, such conditions are best phrased in terms of edge probabilities.
Let $H_r(n, \pi)$ be the random $r$-uniform hypergraph on the vertex set $V=[n]$ where each of the $N_r = {n \choose r}$ possible hyperedges of size $r$ is present independently with probability $\pi$. The binomial random graph in this notation is then $G(n, p)=H_2(n, p)$. In 1979, Shamir asked the following natural question, as reported by Erd\H{o}s \cite{erdHos1981combinatorial}:
\begin{question}
    How large does $\pi=\pi(n)$ need to be for $H_r(n, \pi)$ to contain a perfect matching whp\footnote{We say that a sequence of events $(E_n)_{n \ge 1}$ holds with high probability (whp) if $P(E_n) \rightarrow 1$ as $n \rightarrow \infty$.}, that is, a collection of $n/r$ vertex-disjoint hyperedges? 
\end{question}
Here and in the following, we implicitly assume $n \in r \mathbb Z_+$ whenever necessary. A closely related question, posed by Ruci\'nski~\cite{rucinski1992matching} and Alon and Yuster~\cite{alon1993threshold}, is: 
 \begin{question}
   For which $p=p(n)$ does the random graph $G(n,p)$ contain a $K_r$-factor whp? 
 \end{question}
 That is, for which $p(n)$ does $G(n,p)$ contain a collection of $n/r$ vertex-disjoint cliques of size $r$, also known as a perfect $r$-clique tiling? In the following, we will also call a copy of $K_r$ an $r$-clique.
 For $r=2$, the two questions are the same --- and thanks to Erd\H{o}s and Rényi \cite{erdos1966existence}, we have known since 1966 that there is a sharp threshold\footnote{Recall that a sequence $p^*=p^*(n)$ is called a \emph{sharp threshold} for a graph property $\mc{P}$, if for all fixed $\epsilon>0$ we have $G(n,p) \notin \mc{P}$ whp if $p(n) <(1-\epsilon)p^*(n)$, and $G(n, p) \in \mc{P}$ whp if $p(n) >(1+\epsilon)p^*(n)$. For a (weak) threshold, the conditions become $p=o(p^*)$ and $p^*=o(p)$, respectively.} for the existence of a perfect matching at $p_0=\frac{\log n}{n}$. The lower bound for this is immediate: At $p=(1-\varepsilon)p_0$, some vertices in the graph are still isolated, so there cannot be a perfect matching. The upper bound relies on Tutte's Theorem, for which there is no known hypergraph analogue. \par
 Hence, for the case $r \ge 3$, these questions remained some of the most prominent open problems in random (hyper-)graph theory. Initial results on perfect matchings in random $r$-uniform hypergraphs were obtained by Schmidt and Shamir \cite{SCHMIDT1983287} - guaranteeing a perfect matching for hypergraphs with expected degree $\omega(\sqrt{n})$, with improvements by Frieze and Janson \cite{frieze1995} to $\omega(n^{\frac{1}{3}})$ and further to $\omega\big(n^{1/(5+2/(r-1))}\big)$ by Kim \cite{kim2003}. For clique factors, even determining the special case of triangle factors proved hard, despite partial results by Alon and Yuster \cite{alon1993threshold}, Ruci\'{n}ski \cite{rucinski1992matching} and Krivelevich \cite{Krivelevich1997TriangleFI}.
 
 Finally, both questions were jointly resolved up to constant factors by Johansson, Kahn and Vu in their seminal paper \cite{johansson2008factors}. It had long been assumed that, as in the case $r=2$, the main obstacle in finding a perfect matching in $H_r(n,\pi)$ were \emph{isolated vertices}, that is, vertices not contained in any hyperedge. 
 In the clique factor setting, the obstacle corresponding to isolated vertices are vertices not contained in any $r$-clique. Let
\[
\pi_0 =\pi_0(r)=  \frac{\log n}{{n-1 \choose r-1}} \quad \text{and} \quad p_0=p_0(r)=\pi_0^{1/{r \choose 2}};
\]
then $\pi_0$ and $p_0$ are known to be sharp thresholds for the properties `minimum degree at least $1$' in $H_r(n,\pi)$ and `every vertex is covered by an $r$-clique' in $G(n,p)$, respectively \cite{devlin2017,janson2000random}. 
Johansson, Kahn and Vu \cite{johansson2008factors} showed that $\pi_0$ and $p_0$ are indeed (weak) thresholds for the existence of a perfect matching in $H_r(n,\pi)$ and for the existence of an $r$-clique factor in $G(n,p)$, respectively.

Recently, Kahn \cite{kahn2019asymptotics} proved that $\pi_0$ is in fact a \emph{sharp} threshold for the existence of a perfect matching in $H_r(n,\pi)$. Indeed, he was able to confirm the conjecture that isolated vertices are essentially the only obstacle, and thereby answer Shamir's question, in the strongest possible sense: \par 
Let $\mathbf{h}_1, \dots, \mathbf{h}_{N_r}$ be a uniformly random order of the hyperedges in ${V \choose r}$, then the \emph{random $r$-uniform hypergraph process} $(H_t^r)_{t=0}^{N_r}$ is given by
\[
H^r_t = \{\mathbf{h}_1, \dots, \mathbf{h}_t\}.
\]
Let
\[
\HitH = \mathrm{min}\{ t : H_t^r \text{ has no isolated vertices}\}
\]
be the hyperedge cover hitting time, i.e., the time $t$ where the last isolated vertex `disappears' by being included in a hyperedge. In the graph case $r=2$, Bollob\'as and Thomason \cite{bollobas1985} proved in 1985 that  this hitting time whp coincides with the hitting time for a perfect matching, showing that as soon as the last vertex is covered by an edge, whp there is a perfect matching.  Kahn \cite{kahn2022hitting} showed that this is indeed also the case when $r \ge 3$:

\begin{theorem}[\cite{kahn2022hitting}]\label{theorem:kahnhitting}
Let $r \ge 3$ and $n \in r \mathbb Z_+$, then whp $H_\HitH^r$ has a perfect matching.
\end{theorem}

Can we get a similarly strong answer to the clique factor question? For $r=3$, the question whether a triangle factor exists in the random graph process as soon as every vertex is covered by a triangle was attributed to Erd\H{o}s and Spencer in \cite[\S5.4]{chung1998erdos}.    
This question seems much harder than its Shamir counterpart because, unlike hyperedges in the random hypergraph, cliques do not appear independently of each other. However, for sharp thresholds it has indeed been possible to \emph{reduce} the clique factor problem to the perfect matching problem, using the following coupling result of Riordan (for $r\ge4$) and the first author (for $r=3$):
\begin{theorem}[\cite{heckel2020random,riordan2022random}]\label{theorem:riordancoupling}
	Let $r \ge 3$. There are constants $\varepsilon(r),\delta(r)>0$ such that, for any $p=p(n) \le n^{-2/r + \varepsilon}$, letting $\pi=p^{r \choose 2}(1-n^{-\delta})$, we may couple the random graph $G=G(n,p)$ with the random $r$-uniform hypergraph $H=H_r(n,\pi)$ so that, whp, for every hyperedge in $H$ there is a copy of $K_r$ in $G$ on the same vertex set.\footnote{In  \cite{heckel2020random,riordan2022random}, Theorem~\ref{theorem:riordancoupling} was given with an unspecified $o(1)$-term in place of $n^{-\delta}$; the formulation above is Remark 4 in \cite{riordan2022random} and in the case $r=3$, an unnumbered remark near the end of~\cite{heckel2020random}.}
\end{theorem}
Together with Kahn's sharp threshold result \cite{kahn2019asymptotics}, the following corollary is immediate.
\begin{corollary}
There is a sharp threshold for the existence of a $K_r$-factor at $p_0$.
	\end{corollary}
  In the same spirit, we wish to transfer Kahn's hitting time theorem, Theorem~\ref{theorem:kahnhitting}, directly to the random graph process setting, showing its clique factor analogue. 
  Such a derivation of the factor result from its Shamir counterpart was believed to be out of reach --- Kahn remarks in \cite{kahn2019asymptotics} that `there seems little chance of anything analogous' for Theorem~\ref{theorem:kahnhitting}, and in \cite{kahn2022hitting} that the connection between the factor version and the Shamir version of the result `seems unlikely to extend to' Theorem~\ref{theorem:kahnhitting}.
  
  One important reason for this is that the original coupling provides merely a one-way bound. While it guarantees a copy of $K_r$ in $G=G(n,p)$ on the same vertex set for every $h$ in $H=H_r(n,\pi)$, we cannot, as observed by Riordan \cite{riordan2022random}, expect to find a corresponding hyperedge of $H$ for every $K_r$ in G, since there we will find roughly $n^{2r-2}p^{2\binom{r}{2}-1}$ pairs of $K_r$ sharing two vertices, which is much larger than the expected number $n^{2r-2}\pi^2$ of pairs of hyperedges of $H$ sharing two vertices.
  A second obstacle is that whenever we do have such a pair of overlapping hyperedges in $H$, the corresponding cliques in $G$ will not appear independently of each other in the associated random graph process --- for example the shared edge could be the last to appear, and then those cliques emerge simultaneously in the random graph process. And indeed, extra cliques and pairs of overlapping cliques do pose a challenge, but as we will show subsequently, they will not appear `near' those candidate vertices which may be among the last vertices to be covered by cliques. \par

  Now, let \[(G_t)_{t=0}^{N_2}\] be the random graph process, which is given by the random $r$-uniform hypergraph process in the special case $r=2$. Define
\[
T_G = \mathrm{min}\{ t : \text{ every vertex in $G_t$ is contained in at least one $r$-clique}\}
\]
 as the hitting time of an $r$-clique cover. Then, to transfer Kahn's hitting time result to the clique factor setting, we need to find a copy of $H^r_\HitH$ within the cliques of $G_{T_G}$. That this can be achieved is our main result:
  \begin{theorem} \label{theorem:finalcoupling}
 	Let $r \ge 3$. We may couple the random graph process $(G_t)_{t=0}^{N_2}$ with the random $r$-uniform hypergraph process  $(H^r_t)_{t=0}^{N_r}$ so that, whp, for every hyperedge in $H^r_\HitH$ there is a clique in $G_{\HitG}$ on the same vertex set. 
 \end{theorem}
\begin{corollary}
Let $r \ge 3$ and $n \in r \mathbb Z_+$, then whp $G_\HitG$ contains a $K_r$-factor.
\end{corollary}
\begin{remark}
One might wonder whether the construction of the coupling in Theorem \ref{theorem:finalcoupling} is such that up to time $\HitH$ and $\HitG$, the hyperedges appear in the same order in the random hypergraph process as they do as $r$-cliques in the random graph process.  The answer to this is `almost': a hyperedge which shares two vertices with another hyperedge in $H_\HitH^r$ may appear at a different time as the corresponding $r$-clique, but whp only about $\log^2 n$ such hyperedges exist in $H_\HitH^r$, and we can match up the orders of all other hyperedges and their corresponding cliques.
\end{remark}

The paper is organised as follows. After some preliminaries in  \S\ref{section:preliminaries}, Theorem~\ref{theorem:finalcoupling} is proved in \S\ref{sec_rc_rcoupling}--\ref{section:puteverythingtogether}. An overview of the proof is given in \S\ref{section:proofoverview}. In \S\ref{section:hypergraphfactor}, we formulate and prove the hypergraph analogue of Theorem~\ref{theorem:finalcoupling}. This will immediately imply that in the random $s$-uniform hypergraph process, whp an $r$-(hyper-)clique factor exists as soon as every vertex is covered by an $r$-clique (where $r>s\ge 3$).

 \section{Preliminaries}\label{section:preliminaries}
	\subsection{Notation} 

In the remainder of the paper, we fix $r \ge 3$ and usually suppress the dependence on $r$ in our notation. For example, we simply write $H_t$ instead of $H_t^r$ and so on. Let $M=N_r={n \choose r}$ and $N=N_2={n \choose 2}$. 

By an \emph{$r$-uniform hypergraph} $H$ on the vertex set $V=[n]$, we mean a subset of ${V \choose r}$, the set of all $r$-subsets of $V$. That is, we will use $H$ as a set (of sets of vertices of size $r$) for convenient notation. For a hypergraph $H$ on the vertex set $V=[n]$ and $v \in [n]$, we use $d(v)$ to denote the degree of $v$ in $H$. In a graph $G$, an \emph{$r$-clique} is a clique on $r$ vertices. We denote by $\cl(G)$ the set of vertex sets from ${V \choose r}$ which span $r$-cliques in $G$ (so $\cl(G)$ is an $r$-uniform hypergraph in the aforementioned sense).

 For two functions $a=a(n)$, $b=b(n)$, we say that $a$ is asymptotically at most $b$, denoted by $a \lesssim b$, if $a(n) \le (1+o(1))b(n)$ as $n \rightarrow \infty$.

Throughout the paper, we fix an arbitrary function $g(n)$ satisfying
\begin{equation} \label{eq:defg}
g(n) = o\big(\log n/\log \log n \big)\quad  \text{ and } \quad g(n) \rightarrow \infty.
\end{equation}

\subsection{Probabilistic tools}
We will need the following standard probabilistic tools.
\begin{theorem}[{Chernoff Bound,\cite[Equ.~2.4, Thm.~2.1]{janson2000random}}] \label{theorem:Chernoff}
For $X \sim \mathrm{Bin} (n, p)$ we have
\begin{align*}
  P(X \ge np+t)&\le\left(\frac{np}{np+t}\right)^{np+t}\left(\frac{n-np}{n-np-t}\right)^{n-np-t}
  \le\exp\left(-\frac{t^2}{2(np+t/3)}\right),& 0&\le t\le n-np,\\
  P(X \le np-t)&\le\left(\frac{np}{np-t}\right)^{np-t}\left(\frac{n-np}{n-np+t}\right)^{n-np+t}
  \le\exp\left(-\frac{t^2}{2np}\right),& 0&\le t\le np.
  \end{align*}
\end{theorem}
\begin{theorem}[{Harris' Inequality, \cite[\S2, Le.~3]{bollobas2006percolation}}]
\label{theorem:harris}
Let $n \in \mathbb Z_+$, $p_1, \dots, p_n \in [0,1]$ and consider a product probability space $\Omega= \prod_{i=1}^n \Omega_i$, where each $\Omega_i$ is $\{0,1\}$ endowed with the Bernoulli distribution with parameter $p_i$. We call an event $A \subseteq \Omega$ an \emph{up-set} if, whenever $\omega \in A$ and $\omega'$ is obtained from $\omega$ by changing one component from $0$ to $1$, then $\omega' \in A$. We call $A$ a \emph{down-set}  if $\Omega \setminus A$ is an up-set.
\begin{enumerate}
    \item[a)] If $A$ and $B$ are both up-sets or both down-sets, then
\begin{align*}
P(A \cap B) \ge P(A)P(B) .
\end{align*}
\item[b)] If $A$ is an up-set and $B$ is a down-set, then
\begin{align*}
P(A \cap B) \le P(A)P(B).
\end{align*}
\end{enumerate}

\end{theorem}

\subsection{The standard coupling and the critical window}\label{section:criticalwindow}

It will sometimes be useful to work with the following standard device which gives a convenient coupling of the random hypergraphs $H(n, \pi)$ for all $\pi \in [0,1]$ as well as the random hypergraph process.

\begin{definition}[Standard coupling]\label{def:mastercoupling}
	For every $h \in {V \choose r}$, let $U_h $ be an independent random variable, uniform from $[0,1]$. Let
	\[
	H_\pi = (V, \{h: U_h \le \pi\}).
	\]
	Then $H_\pi \sim H(n,\pi)$. Furthermore, almost surely all values $U_h$, $h \in {V \choose r}$, are distinct. If they are distinct, they give an instance of the random hypergraph process $(H_t)_{t=0}^{M}$ in a natural way, as we can add the hyperedges in ascending order of $U_h$.
\end{definition}

For much of the proof, we will operate within the following \emph{critical window}: Define $\pi_-$ and $\pi_+$ by setting
\begin{equation} \label{eq:defcriticalwindow}
\pi_\pm =\frac{\log n \pm g(n)}{{n-1 \choose r-1}},
\end{equation}
where $g(n)$ is the function which was fixed globally in \eqref{eq:defg}, and let 
\begin{equation} \label{eq:defcriticalwindow_p}
    p_\pm= \big(\pi_\pm/(1-n^{-\delta})\big)^{1/\binom{r}{2}},
\end{equation}
where $\delta$ is the constant from Theorem~\ref{theorem:riordancoupling}. Note that for $n$ large enough we have $p_+ \le n^{-2/r+\epsilon}$, where $\epsilon$ is the constant from Theorem~\ref{theorem:riordancoupling}, so we may apply Theorem~\ref{theorem:riordancoupling} with $p=p_+$ and $\pi=\pi_+$ later.

It is a well-known fact that $(\pi_-, \pi_+)$ is the `critical window' for the disappearance of the last isolated vertex in a random $r$-uniform hypergraph  (see \cite[Lemma 5.1(a)]{devlin2017}), and so by Theorem~\ref{theorem:kahnhitting} for the appearance of a perfect matching: Whp, $H(n, \pi_-)$ does not contain a perfect matching, while $H(n, \pi_+)$ contains a perfect matching whp. Similarly, $(p_-, p_+)$ is the critical window for the disappearance of the last vertex not covered by an $r$-clique in a random graph (see \cite[Theorem 3.22]{janson2000random}). So if we couple as in Definition~\ref{def:mastercoupling}, then whp we have
\begin{align}
G_{p_-} \subset G_\HitG \subset G_{p_+}\quad \text{ and } \quad H_{\pi_-} \subset H_\HitH \subset H_{\pi_+}. \label{eq:mastercoupling}
\end{align}

\subsection{Bad events}

In this section we state certain undesirable properties, and show that $H \sim H(n, \pi_+)$ does not have them whp. We first need some terminology. For an $r$-uniform hypergraph $H$, let
\[
n(H)= (r-1)e(H)+c(H)-v(H)
\]
be the \emph{nullity} of $H$, where $e(H)$, $c(H)$, $v(H)$ denote the number of hyperedges, components and vertices of $H$, respectively. 
\begin{definition}[{\cite[Def.~11]{riordan2022random}}]\label{def:avoidable}
We call a connected $r$-uniform hypergraph $H$ on at most $2^{r+1}$ hyperedges with $n(H)\ge 2$ an \emph{avoidable configuration}.
	\end{definition}
\begin{definition}We say that a vertex $v$ in a hypergraph on $n$ vertices is a \emph{low-degree vertex} if $d(v)\le 7g(n)$, where $g(n)$ is the function that we fixed globally in \eqref{eq:defg}.
 \end{definition}
\begin{definition}We say that two hyperedges in a hypergraph are \emph{partner hyperedges} if they share exactly two vertices. \end{definition}

Now we are ready to state the \emph{bad events}. The first two are essentially the bad events from \cite{riordan2022random}. Note that the event $ \Bone = \Bone(\pi)$ depends on the choice of $\pi$.

\begin{itemize}
\item[] $\Bone$: \quad There is a vertex of degree more than $\binom{n-1}{r-1}\pi +\max(\binom{n-1}{r-1}\pi,3\log n)$.
\item[]$\Btwo$:  \quad $H$ contains an {avoidable configuration}.
\item[]$\Bthree$: \quad There are more than $(\log n)^{8 g(n) }$  low-degree vertices.
\item[]$\Bfour$: \quad There are more than $\log^3 n$ {pairs of partner hyperedges}.
\item[]$\Bfive$: \quad There is an isolated vertex.
\end{itemize}
Let $\mathcal{B} = \Bone(\pi_+) \cup \Btwo \cup \Bthree \cup \Bfour \cup \Bfive $. 
	
\begin{lemma} \label{lemma:Hnotbad}Let $H \sim H(n, \pi)$ with $\pi \le n^{1-r+o(1)}$ ,
then whp $H \notin  \Bone(\pi) \cup \Btwo$. If further $\pi=\pi_+$, then whp $H \notin \mathcal{B}$.
\end{lemma} 
\begin{proof}
$\mc B_1$ and $\mc{B}_3$ follow from the Chernoff bounds (Theorem ~\ref{theorem:Chernoff}), noting that the expected degree of a given vertex is of order $n^{r-1}\pi$, and the union bound and Markov's inequality, respectively. For $\Bfour$, note that the expected number of pairs of partner hyperedges in $H(n, \pi_+)$ is of order $\log^2 n$, and apply Markov's inequality. It was shown in \cite[Lemma 12]{riordan2022random} that whp $\mc B_2$ does not hold. Finally, $\mc B_5$ was mentioned previously in \S\ref{section:criticalwindow}, noting that $\pi_+$ is the end of the `critical window' for the disappearance of the last isolated vertex.
\end{proof}

\begin{remark}\label{remark:vertexdisjoint}Note that two hyperedges overlapping in three or more vertices comprise an avoidable configuration. Furthermore, a hyperedge with more than one partner hyperedge is an avoidable configuration. Moreover, two pairs of partner hyperedges that share a vertex are an avoidable configuration. Therefore, if $H \notin \mc{B}$, no two hyperedges share more than two vertices, each hyperedge has at most one partner hyperedge, and all pairs of partner hyperedges are vertex-disjoint. 
\end{remark}

We will also need the following technical lemma, which is proved in the appendix.
\begin{lemma}\label{lemma:lowdegreenotbad}
    Let $K$ be a fixed hypergraph such that $E[X_K]\le n^{o(1)}$, where $X_K$ is the number of  copies of $K$ in $H = H(n, \pi_+)$. Then whp no low-degree vertex of $H$ is contained in a copy of $K$. 
\end{lemma}

\subsection{Proof overview}\label{section:proofoverview}
Define $p_+, \pi_+$ as in equations \eqref{eq:defcriticalwindow} and \eqref{eq:defcriticalwindow_p}. Our starting point is the coupling of $G \sim G(n,p_+)$ and $H \sim H(n, \pi_+)$ given by Theorem~\ref{theorem:riordancoupling}. We review this coupling in \S\ref{sec_rc_rcoupling}--\ref{section:r3algo}, and then analyse it more in-depth in \S\ref{section:extracliques}--\ref{section:lowdegreeextra}. As in \cite{heckel2020random,riordan2022random}, the case $r=3$ requires special treatment, and the reader may find it useful to focus on the case $r \ge 4$ at first and skip over the parts related to $r=3$ (\S\ref{section:r3algo} and \S\ref{section:extracliques3}). The goal of \S\ref{sec_rc_rcoupling} is to prove that whp $G$ does not have any `extra' $r$-cliques, meaning cliques which are not represented by hyperedges in $H$, that are incident with low-degree vertices. This will ensure later that in the construction of the coupled processes, the `extra' cliques in $G$ do not affect the hitting time $\HitG$.

In \S\ref{section:processcoupling1}, which is the heart of the proof, we take the coupled $G \sim G(n,p_+)$ and $H \sim H(n, \pi_+)$ and proceed by carefully coupling uniform orders of the edges of $G$ and hyperedges of $H$. Since $p_+$ and $\pi_+$ are the upper ends of the respective critical windows (see \S\ref{section:criticalwindow}), whp this couples (copies of) the stopped graph process $G_\HitG$ and the stopped hypergraph process $H_\HitH$. We will show that this coupling \emph{almost} does what we want: for all hyperedges $h \in H_\HitH$, except those in a small exceptional set $\mc{E}$, there is an $r$-clique in $G_\HitG$ on the same vertex set. Moreover, we show that whp all $h \in \mc{E}$ have a partner hyperedge which appears between time $\HitH$ and time $\HitH+\left\lfloor g(n)n \right\rfloor$. 

To prove Theorem \ref{theorem:finalcoupling}, we now only need to show that we can get rid of the hyperedges in $\mc{E}$ and still have an instance of the stopped random hypergraph process. To this end, in \S\ref{section:step2} we show that $\mc{E}$ can be whp embedded into a binomial random subset $\mc{R} \subset H_\HitH$ where each hyperedge $h \in H_\HitH$ is included independently with a small probability. In \S\ref{section:step3}, we show that if we remove the hyperedges in $\mc{R}$ from the hypergraph process up to time $\HitH$, whp this essentially does not change the hitting time $\HitH$, and in particular whp $H_\HitH \setminus \mc{R}$ is still an instance of the stopped random hypergraph process.

In \S\ref{section:puteverythingtogether}, we put all the different pieces together, following through the chain of couplings to construct the coupling in Theorem~\ref{theorem:finalcoupling}.

\section{Coupling of $G(n,p)$  and $H(n,\pi)$}\label{sec_rc_rcoupling}
We start this section in \S\ref{section:riordanalgo} with a brief review  of Riordan's coupling from Theorem~\ref{theorem:riordancoupling} for $r \ge 4$.  Our description and notation will largely follow that in \cite{riordan2022random}, adapted slightly to our needs. In \S\ref{section:r3algo}, we describe the modifications to Riordan's coupling for $r=3$ which were made by the first author in \cite{heckel2020random}, and establish some basic properties. In \S\ref{section:extracliques}---\S\ref{section:lowdegreeextra}, we proceed with a more in-depth analysis of the `extra' $r$-cliques in $G$ which are not represented by hyperedges in $H$. In particular, in Lemma \ref{lemma:extrabound} (for $r \ge 4$) and Lemma \ref{lemma:extrabound3} (for $r=3$) we give an upper bound for the probability that a given $r$-set of vertices which is not a hyperedge in $H$ spans an $r$-clique in $G$. Such a bound was not relevant for the construction of the one-sided coupling in \cite{riordan2022random}, but it will be crucial for the hitting time version of the coupling.
Recall that a vertex $v$ is called \emph{low-degree} if $d(v) \le 7g(n)$.
The following result is proven in \S\ref{section:lowdegreeextra}.
	\begin{lemma}\label{lemma:lowdegreenoextra}
	Couple $G \sim G(n, p_+)$ and $H\sim H(n, \pi_+)$ via the coupling described in \S\ref{section:riordanalgo}--\ref{section:r3algo}. We call the hyperedges in $\mathrm{cl}(G) \setminus H$ \emph{extra cliques}. Then whp, no low-degree vertex of $H$ is incident with any extra clique in $G$.
	\end{lemma} 
	This lemma will be used later on to show that extra cliques in the random graph process do not influence the hitting time $\HitG$, as none of them are incident with any of the candidate vertices which are amongst the last to be covered by $r$-cliques.

As in \cite{heckel2020random,riordan2022random}, we will let 
\[p=p(n) \le n^{-2/r+o(1)}\]
in order to simplify calculations, although upon closer inspection the $o(1)$-term could be replaced with a small constant $\epsilon=\epsilon(r)$. We let 
\[\pi=(1-n^{-\delta})p^{r \choose 2},\]
where $\delta$ is the constant from  Theorem~\ref{theorem:riordancoupling}. Most of \S\ref{sec_rc_rcoupling} will apply to these general $p$ and $\pi$, except for the proof of Lemma~\ref{lemma:lowdegreenoextra} in \S\ref{section:lowdegreeextra} where we set $p=p_+$ and $\pi=\pi_+$.

As was the case in \cite{heckel2020random,riordan2022random}, the case $r=3$ requires a considerable amount of extra work due to complications arising from the presence of so-called `clean $3$-cycles'. It may therefore be helpful to focus on the case $r \ge 4$ initially, and skip over \S\ref{section:r3algo} and \S\ref{section:extracliques3} which relate to the case $r=3$.

\subsection{The coupling algorithm for {$r\ge 4$}} \label{section:riordanalgo}
Order the $M={n \choose r}$ potential hyperedges in some arbitrary way as $h_1, \dots, h_{M}$, and for $1\le j \le M$, let $A_j$ be the event that there is an $r$-clique in $G \sim G(n,p)$ on the vertex set of $h_j$. We construct the coupling of $G \sim G(n,p)$ and $H \sim H(n, \pi)$ step by step; in step $j$ revealing whether or not $h_j \in H$, as well as some information about $A_j$.

\noindent \textbf{Coupling algorithm:} For each $j$ from $1$ to $M$:
\begin{itemize}
	\item Calculate $\pi_j$, the conditional probability of $A_j$ given all the information revealed so far.
	\item If $\pi_j \ge \pi$, toss a coin which lands heads with probability $\pi/\pi_j$, independently of everything else. If the coin lands heads, then test whether $A_j$ holds (which it does with probability exactly $\pi_j$). Include the hyperedge $h_j$ in $H$ if and only if the coin lands heads and $A_j$ holds. (Note that the probability of including $h_j$ is exactly $\pi/\pi_j \cdot \pi_j = \pi$.)
	\item If $\pi_j < \pi$, then toss a coin which lands heads with probability $\pi$ (independently of everything else), and declare $h_j$ present in $H$ if and only if the coin lands heads. If this happens for any $j$, we say that the coupling has failed.
	\end{itemize}
After going through steps $j=1, \dots, M$, we have decided all hyperedges of $H$, and revealed information on the events $A_1, \dots, A_M$ of $G$. Now choose $G$ conditional on the revealed information on the events $A_j$. Clearly, this algorithm generates the correct distributions of $H(n, \pi)$ and $G(n, p)$, and in \cite{riordan2022random} it is shown that for an appropriate choice of constants $\epsilon(r), \delta(r)>0$, whp the algorithm does not fail. If it does fail, then 
\[H \in \mc{B}_1 \cup \mc{B}_2 \quad \big( \,\, =  \mc{B}_1(\pi) \cup \mc{B}_2 \,\,\big).\]

The algorithm, if it succeeds, generates the answers yes/no for each $h_j$, and it generates answers yes/no/$*$ for $A_j$, where $*$ means that we did not decide $A_j$. The latter happens either if $\pi_j \ge \pi$ and the coin lands tails, or if $\pi_j<\pi$. 

So for $0 \le j \le M$, let $\bY_j, \bN_j, \bS_j  \subset [j]$ be the sets of indices $i \le j$ where the algorithm gives the answer `yes', `no', and $*$, respectively, for the event $A_i$. Note that these sets are disjoint and their union is $[j]$.

In particular, before we decide the hyperedge $h_j$ and the event $A_j$, we have exactly the information $\bigcap_{i \in \bY_{j-1}}A_i \cap \bigcap_{i \in \bN_{j-1}}A_i^c$ on $G$, and so
\[
\pi_j = P\Big(A_j \mid \bigcap_{i \in \bY_{j-1}}A_i \cap \bigcap_{i \in \bN_{j-1}}A_i^c\Big).\]
Furthermore, note that if the coupling does not fail, we have
\[\bY_j=\{i \le j: h_i \in H\} \quad \text{ and } \quad \bY_M= \{i: h_i \in H\}.\]
Let
\[
\histj = (\bY_j, \bN_j, \bS_j),\]
which encodes the relevant history of the algorithm up to and including time $j$.

\subsection{Modified coupling algorithm for {$r=3$}}\label{section:r3algo}

Riordan's proof does not extend to the case $r=3$ because of a single problematic hyperedge configuration:  clean $3$-cycles - which have nullity $1$ and therefore do not constitute an avoidable configuration.

\begin{definition}\label{def:clean3cycles}A \emph{clean $\mathit{3}$-cycle} is a set of three $3$-uniform hyperedges where each pair meets in exactly one distinct vertex. In a slight abuse of notation, we will also call an edge configuration in a graph where each such hyperedge is replaced by a triangle a clean $3$-cycle. In such an edge configuration, we call the copy of $K_3$ spanned by the three vertices where the pairs of hyperedges meet the \emph{middle triangle} of the clean $3$-cycle.
\end{definition}

The key observation in the modified coupling from \cite{heckel2020random} is that there are very few clean $3$-cycles in both $H$ and $G$, and that their distributions are essentially the same, which is Lemma~\ref{lemma:dtv} below. So it is possible to choose and match up the clean $3$-cycles in $H$ and $G$ first, and then run Riordan's coupling algorithm conditional on this choice of clean $3$-cycles (so that no further clean $3$-cycles can appear during the execution of the algorithm, which circumvents the specific problem in Riordan's proof for the case $r=3$).

\begin{lemma}[{\cite[Le.~4]{heckel2020random}}]\label{lemma:dtv}
For $p \le n^{-2/3+o(1)}$ and $\pi=(1-n^{-\delta})p^3$ as before, let $\mathcal{C}_G$ and $\mathcal{C}_H$ be the collections of clean $3$-cycles in a random graph $G \sim G(n,p)$ and in a random hypergraph $H \sim H_3(n, \pi)$, respectively. Then we can couple $\mathcal{C}_G$ and $\mathcal{C}_H$ so that, whp, $\mathcal{C}_G=\mathcal{C}_H$.
\end{lemma}

\noindent \textbf{Modified coupling algorithm:} 
\begin{itemize}
\item Choose the collections $\mathcal{C}_G$ and $\mathcal{C}_H$ of clean $3$-cycles in $G$ and in $H$, coupling as in Lemma \ref{lemma:dtv} so that whp
\[
\mathcal{C}_G =\mathcal{C}_H.
\] If $\mathcal{C}_G \neq \mathcal{C}_H$, we say that the coupling has failed. 
\item[] Let $\mc{U}_G$ and $\mc{U}_H$ be the events (which are up-sets) that $G$ and $H$ contain the edges and hyperedges contained in $\mc{C}_G$ and $\mc{C}_H$, respectively. Let $\mc{D}_G$ and $\mc{D}_H$ be the events (which are down-sets) that $G$ and $H$ contain no other clean $3$-cycles.

Now as before, we order the $M'={n \choose 3}-|\{h\in c:c\in\mc{C}_H\}|$ remaining potential hyperedges in some arbitrary way as $h_1, \dots, h_{M'}$, and for $1\le j \le M'$ we let $A_j$ be the event that there is a triangle in $G \sim G(n,p)$ on the vertex set of $h_j$. 

\item[] Now for each $j$ from $1$ to $M'$:
	\item Calculate $\pi_j$, the conditional probability of $A_j$ given $\mathcal{U}_G$, $\mathcal{D}_G$ and all the information revealed on $G$ so far.
	\item Calculate $\pi_j'$, the conditional probability of the hyperedge $h_j$ being present in $H$, given $\mathcal{U}_H$, $\mathcal{D}_H$ and all the information revealed on $H$ so far.
	\item Case 1: $\pi_j \ge \pi'_j$. If $\pi_j=\pi'_j=0$, then we do not include $h_j$ in $H$, and we decide that $A_j$ does not hold. Otherwise we have $\pi_j>0$, and in that case we toss a coin which lands heads with probability $\pi_j'/\pi_j$, independently of everything else. If the coin lands heads, then test whether $A_j$ holds (which it does with probability exactly $\pi_j$). Include the hyperedge $h_j$ in $H$ if and only if the coin lands heads and $A_j$ holds. (Note that the probability of including $h_j$ is exactly $\pi_j'/\pi_j \cdot \pi_j = \pi_j'$.)
	\item Case 2: $\pi_j < \pi_j'$. Toss a coin which lands heads with probability $\pi_j'$ (independently of everything else), and declare $h_j$ present in $H$ if and only if the coin lands heads. If this happens for any $j$, we say that the coupling has failed.
	\end{itemize}
	After going through the steps $j=1, \dots, M'$, we have decided all hyperedges of $H$, and revealed information on the events $A_1, \dots, A_{M'}$ of $G$. Now choose $G$ conditional on the revealed information on the events $A_j$ and on $\mathcal{C}_G$ and $\mathcal{D}_G$. Clearly this algorithm generates the correct distributions $H \sim H_3(n, \pi)$ and $G \sim G(n, p)$, and in \cite{heckel2020random} it is shown that for an appropriate choice of constants $\epsilon, \delta>0$, whp the algorithm does not fail. If it does fail, then 
	\[\mc{C}_G \neq \mc{C}_H \text{ or } H \in \mc{B}_1 \cup \mc{B}_2.\]

Similarly to \S\ref{section:riordanalgo}, we let
\[
\histj = (\bY_j, \bN_j, \bS_j)\]
encode the history of the algorithm up to and including time $j$, where $\bY_j$ and $\bN_j$ are the sets of all $i \in [j]$ where we have received the answers `yes' and `no' for the event $A_i$, respectively, and $\bS_j=[j] \setminus (\bY_j \cup \bN_j)$. Then, when deciding $h_j$ and $A_j$, the information we have on $H$ and $G$ is given by $\mathbf{Hist}_{j-1}$ and by $\mc{C}_G$, $\mc{C}_H$, with the corresponding events
\[\mc{U}_G, \mc{U}_H, \mc{D}_G, \mc{D}_H\]
that the clean $3$-cycles in $\mc{C}_G$ and $\mc{C}_H$ are present, and no others. So
\begin{equation}
\pi_j = P \Big(A_j \Big|\, \mc{U}_G \cap \mc{D}_G \cap \bigcap_{i \in \bY_{j-1}} A_i \cap \bigcap_{i \in \bN_{j-1}} A_i^c \Big) \label{eq:defpij}
\end{equation}
and, letting $B_j = \{h_j \in H\}$, unless the coupling failed at an earlier step,
\begin{equation}
\pi_j' = P \Big(B_j \Big|\, \mc{U}_H \cap \mc{D}_H \cap \bigcap_{i \in \bY_{j-1} } B_i \cap \bigcap_{i \in \bN_{j-1} \cup \bS_{j-1}} B_i^c \Big). \label{eq:defpijprime}
\end{equation}
We make a number of  observations in the following lemma which is proved in the appendix.
\begin{lemma}\label{lemma:boundforpijprime}
 Suppose that we are at time $j$ in the coupling algorithm for $\pi=(1-n^{-\delta})p^{3}$, before making decisions on $h_j \in H$ and $A_j$, and that $\tilde H \notin \mc B_1\cup\mc B_2$, where $\tilde H$ is the hypergraph containing exactly the hyperedges $\{h_i : i \in \bY_{j-1}\}$ and those from the clean $3$-cycles in $\mc{C}_H$. Then all of the following hold. 
 \begin{itemize}
      \item[a)] If $\pi_j'=0$, then $\pi_j=0$.
    \item[b)] $\pi_j' \le \pi$.
        \item[c)] Suppose that $\pi_j > \pi_j'$, then
    \[
\pi_j' = \pi  (1+n^{-1+o(1)}).
\]

 \item[d)] If $\pi_j < \pi_j'$, then $\pi_j=0$.
    \end{itemize}
\end{lemma}

\subsection{Extra cliques, $r \ge 4$} \label{section:extracliques}
We say that an $r$-clique of $G$ is an \emph{extra clique} if there is no hyperedge in $H$ on the same vertex set. So the extra cliques are exactly the elements of $\cl(G)\setminus H$. In the following lemma, we will bound the probability that the clique corresponding to a given $h_j$ becomes an extra clique in $G$ for $r \ge 4$. The case $r=3$ will be treated separately in \S\ref{section:extracliques3}.

\begin{lemma}\label{lemma:extrabound}
Suppose that $r \ge 4$, let $\pi=(1-n^{-\delta})p^{\binom{r}{2}}$, $j \in [M]$, and let $H_0$ be a hypergraph so that $h_j \notin H_0$ and $H_0 \notin \mathcal{B}_1 \cup \mathcal{B}_2$. 
Set
	\[
\pi_j^* = P\big(A_j \mid \bigcap_{i: h_i \in H_0} A_i\big).
\]
Then, if the coupling in Theorem~\ref{theorem:riordancoupling} produces the hypergraph $H=H_0$, we can bound the probability that the clique on the vertex set of $h_j$ is  an extra clique in $G$ by
\[
P\big(A_j \mid H=H_0 \big) \le \frac{\pi_j^*-\pi}{1-\pi}.
\]
		\end{lemma}

\begin{proof}
	Condition on $H=H_0$. First of all, note that as $H_0 \notin \mathcal{B}_1 \cup \mathcal{B}_2$, the coupling does not fail. Consider the coupling algorithm after we have decided on the final hyperedge $h_M$ (and possibly on $A_M$), but before generating the final graph $G$. In particular, we know that $H=H_0$, and (as the coupling does not fail) we have $\bY_M=Y_{M}:=\{i: h_i \in H_0\}$. We also know $\bY_i=Y_i := Y_{M} \cap [i]$ for all $0 \le i \le M$. However, given only the information $H=H_0$, we do not know what the sets $\bN_i$ and $\bS_i$ are, $0 \le i \le M$. In the following, we want to find bounds which hold for all possible instances of these sets, so we denote by $S_{i}$ and $N_{i}$ specific instances of the random sets. In other words, for all $0 \le i \le M$, we consider disjoint sets $S_{i}$ and $N_{i}$ whose union is $[i]\setminus Y_{i}$, and we let $\text{Hist}_{i}=(Y_{i}, N_{i},S_{i})$.

We start by observing that, for a given $\text{Hist}_{M}$,
\begin{align}
P(A_j \mid \mathbf{Hist}_M=\text{Hist}_{M}) &= P\big(A_j \mid \bigcap_{i \in Y_{M}}A_i \cap \bigcap_{i \in N_{M}}A_i^c\big) \nonumber \\
&\le P\big(A_j \mid \bigcap_{i \in Y_{M}}A_i \big) = \pi^*_j, \label{eq:bound1}
\end{align}
where the inequality follows by Harris' Lemma (Theorem \ref{theorem:harris}), since $A_j$ is an up-set, $\bigcap_{i \in N_{M}} A_i^c$ is a down-set and $G(n,p)$ conditional on the principal up-set  $\bigcap_{i \in Y_{M}} A_i$ is still a product probability space.

Since the inequality above holds for every possible $\text{Hist}_{M}$, it also holds if we condition only on the partial information $\mathbf{Hist}_{j-1}=\mathrm{Hist}_{j-1}$, $\bS_{j}=S_{j-1}\cup \{j\}$ and $H=H_0$ (since we can consider every possible continuation of these $\bY_{j-1}$, $\bN_{j-1}$, $\bS_j$ to a $\text{Hist}_{M}$)\footnote{To be precise, Equation \eqref{eq:bound1} implies Equation \eqref{eq:part1} for the following reason. Let $A, B$ be some events and $c\in[0,1]$. Then, if for some countable partition $B= \bigcup_i B_i$ of $B$, we have $P(A \mid B_i) \le c$ for every $i$, then $P(A \cap B) = \sum_i P(A \cap B_i) = \sum_i P(B_i)P(A \mid B_i) \le c \sum_i P(B_i)= c P(B)$, so $P(A \mid B) \le c$.}:
\begin{equation}
	P(A_j | \mathbf{Hist}_{j-1} = \text{Hist}_{j-1}, \bS_{j}=S_{j-1}\cup \{j\}, H=H_0) \le \pi^*_j. \label{eq:part1}
\end{equation}

Furthermore, given a fixed $\text{Hist}_{j-1}$, it follows similarly to the above that
\begin{align}
\pi_j = 
P\big(A_j | \bigcap_{i \in Y_{j-1}}A_i \cap \bigcap_{i \in N_{j-1}}A_i^c\big) \le  P\big(A_j | \bigcap_{i \in Y_{j-1}}A_i \big) \le P\big(A_j | \bigcap_{i \in Y_{M}}A_i \big) = \pi_j^*, \label{eq:part2}
\end{align}
where the second inequality follows from the fact that $Y_{j-1} \subset Y_{M}$.

	The advantage of bounding the two probabilities in \eqref{eq:part1}, \eqref{eq:part2} by $\pi_j^*$ is that $\pi_j^*$ only depends on $H_0$ and not on the full $\text{Hist}_{j-1}$ or $\text{Hist}_{M}$.
	
	In the final graph $G$, the event $A_j$ can only hold if $j \in S_M$, because $j \notin Y_M$ since $h_j \notin H_0$, and if $j$ were in $N_M$ then $A_j^c$ would hold.
	This implies $\bS_{j}=\bS_{j-1}\cup\{j\}$. Consider some $\text{Hist}_{j-1}$, and suppose that $\text{Hist}_{j-1}$ is such that $\pi_j \ge \pi$. Then we have $\bS_{j}=\bS_{j-1}\cup\{j\}$ if and only if the coin in step $j$ of the algorithm lands tails. So, with \eqref{eq:part2}, we obtain
	\begin{equation}\label{eq:bdd1}
		P(\bS_{j}=S_{j-1}\cup\{j\} \mid \mathbf{Hist}_{j-1}=\text{Hist}_{j-1}, H=H_0) =  \frac{1-\pi/\pi_j}{1-\pi} \le \frac{1-\pi/\pi_j^*}{1-\pi},
	\end{equation}
	because $1-\pi/\pi_j$ is the probability that the coin lands tails, and $1-\pi$ is the overall probability that we do not add $h_j$ to $H$, on which we have conditioned in $H=H_0$ (note that by construction of the algorithm, the decisions on $h_i \in H$ for $i >j$ are independent from what happens at step $j$ of the algorithm or earlier, so only the information $h_j \notin H$ from conditioning on $H=H_0$ is relevant for the event $\bS_{j}=S_{j-1}\cup\{j\}$). 
	
	So, given some $\text{Hist}_{j-1}$ such that $\pi_j \ge \pi$, 	using \eqref{eq:part1} and \eqref{eq:bdd1} it follows that
\begin{align}		P \left( A_j \mid \mathbf{Hist}_{j-1}=\text{Hist}_{j-1}, H=H_0 \right) &=	P \left( A_j \cap \{\bS_{j}=S_{j-1}\cup\{j\} \}\mid \mathbf{Hist}_{j-1}=\text{Hist}_{j-1}, H=H_0\right) \nonumber \\
			 &= P \left(\bS_{j}=S_{j-1}\cup\{j\} \mid \mathbf{Hist}_{j-1}=\text{Hist}_{j-1}, H=H_0\right) \nonumber \\
			& \quad \cdot P \left( A_j \mid \mathbf{Hist}_{j-1}=\text{Hist}_{j-1}, \bS_{j}=S_{j-1}\cup\{j\}, H=H_0 \right) \nonumber \\	
		&\le  \frac{1-\pi/\pi_j^*}{1-\pi} \cdot \pi_j^* = \frac{\pi_j^*-\pi}{1-\pi}. \label{eq:bd3}
	\end{align}

	Furthermore, note that \eqref{eq:bd3} also holds for any $\text{Hist}_{j-1}$ such that $\pi_j < \pi$: In this case, by a remark shortly after equation (7) in \cite{riordan2022random}, we have $\pi_j=0$. This means that the probability on the left-hand side of \eqref{eq:bd3} is $0$ (already after conditioning on $\mathbf{Hist}_{j-1}=\text{Hist}_{j-1}$, the event $A_j$ has probability $0$), whereas  the right-hand-side of \eqref{eq:bd3} is nonnegative because $\pi^*_j \ge P(A_j)= p^{r \choose 2} \ge \pi$. So since \eqref{eq:bd3} holds conditional on every possible $\mathbf{Hist}_{j-1}=\text{Hist}_{j-1}$, it also holds when only conditioning on  $H=H_0$:
\[
P(A_j \mid H=H_0) \le \frac{\pi_j^*-\pi}{1-\pi}.
\]
	\end{proof}
	
	\subsection{Extra cliques, $r=3$}\label{section:extracliques3}
	As before, we let $p\le n^{-2/3+o(1)}$ and set $\pi=p^{r \choose 2}(1-n^{-\delta(3)})$. The following lemma corresponds to Lemma~\ref{lemma:extrabound} for the case $r \ge 4$.
			\begin{lemma}\label{lemma:extrabound3}
Suppose $r = 3$, let $j \in [M']$, and let $H_0$ be a hypergraph so that $h_j \notin H_0$, and $H_0 \notin \mathcal{B}_1 \cup \mathcal{B}_2$.  Set
	\[
\pi_j^* = P\big(A_j \mid \bigcap_{i:h_i \in H_0} A_i\big),
\]
Let $\mc{F}$ be the event that the coupling in Theorem~\ref{theorem:riordancoupling} for $r=3$ fails in the first step (i.e.\ that $\mc{C}_G \neq \mc{C}_H$). Then, if $\mc{F}$ does not hold and the coupling produces the hypergraph $H=H_0$, we can bound the probability that $A_j$ is  an extra clique in $G$:
\[
P\big(A_j \mid H=H_0\,,\, \mc{F}^c \big) \le \frac{\pi_j^*-\pi}{1-\pi}+n^{-3+o(1)}
\]
		\end{lemma}
		\begin{proof}
		         First of all, note that $\{H=H_0\}\cap \mc{F}^c$ implies that the coupling does not fail, as $H_0 \notin \mathcal{B}_1 \cup \mathcal{B}_2$. Let $\mc{K}_G$ and $\mc{K}_H$ denote the events that $G$ and $H$ contain exactly the clean $3$-cycles of $H_0$. Then the coupling is set up so that the distribution of $G$, conditional on $H=H_0$ and $\mc{F}^c$, is the same as the distribution of $G$ conditional on $H=H_0$ and $\mc{K}_G$ (note that $H=H_0$ implies $\mc{K}_H$). In particular,
		        \begin{equation} 
		       \label{eq:equivalentcondition} 
		       P\big(A_j \mid H=H_0\,,\, \mc{F}^c \big) = P\big(A_j \mid H=H_0\,,\, \mc{K}_G\big)
		        .\end{equation}
		        
		        As in the proof of Lemma~\ref{lemma:extrabound}, we will consider different histories $\textbf{Hist}_i=\text{Hist}_i=(Y_i, N_i, S_i)$ of the process which are compatible with the outcome $H=H_0$. As the coupling does not fail, we have $\bY_{M'}=Y_{M'}:=\{i: h_i \in H_0\}$, which gives us that $\bY_i=Y_i := Y_{M'} \cap [i]$ for all $0 \le i \le M'$. Given only the outcome $H=H_0$ (and $\mc{F}^c$), we do not know what the sets  $\bN_i$ and $\bS_i$ are, $0 \le i \le M'$, so we will consider bounds which are valid over all possible choices of these random sets, and we denote specific instances of the sets by $S_{i}$ and $N_{i}$. In other words, for all $0 \le i \le M'$, we fix some disjoint sets $S_{i}$ and $N_{i}$ whose union is $[i]\setminus Y_{i}$, and we let $\text{Hist}_{i}=(Y_{i}, N_{i},S_{i})$.

		        
		        Note that if $\text{Hist}_{j-1}$ is so that $\pi_j < \pi_j'$, then $\pi_j=0$ by Lemma~\ref{lemma:boundforpijprime}, thus the probability of $A_j$ is $0$ after step $j-1$ of the algorithm, and it follows that $P(A_j \mid H=H_0, \mc{F}^c, \pi_j < \pi_j')=0$. Hence, we can assume in the following that $\pi_j \ge \pi_j'$. Further, note that if $\pi_j = \pi_j'$, then in step $j$ of the algorithm we have decided that $h_j \notin H$ if and only if $j \in N_j$ (that is, we decide that $A_j$ does not hold). So, as $h_j \notin H$, we have $P(A_j \mid H=H_0, \mc{F}^c, \pi_j = \pi_j')=0$. Thus, indeed we only need to bound $P(A_j \mid H=H_0, \mc{F}^c, \pi_j > \pi_j')$, and for the rest of the proof we only consider histories $\text{Hist}_{j-1}$ so that $\pi_j > \pi_j'$.

Let $\mc{U}_G$, $\mc{D}_G$ be the events that $G$ contains the clean $3$-cycles of $H_0$, and no other clean $3$-cycles, respectively, so $\mc{K}_G = \mc{U}_G \cap \mc{D}_G$. 		        Similarly to \eqref{eq:bound1},  we observe that 
\begin{align*}
P(A_j \mid  \mc{K}_G\,,\, \mathbf{Hist}_{M'}=\text{Hist}_{M'}) &= P\big(A_j \mid  \mc{U}_G\cap \mc{D}_G \cap \bigcap_{i \in Y_{M'}}A_i \cap \bigcap_{i \in N_{M'}}A_i^c\big) \nonumber \\
&\le P\big(A_j \mid \bigcap_{i:h_i \in H_0}A_i \big) = \pi^*_j
\end{align*}
(noting that $\mc{U}_G \cap \bigcap_{i \in Y_{M'}}A_i =\bigcap_{i:h_i \in H_0} A_i$ is a principal up-set). Since this bound holds for every possible $\text{Hist}_{M'}$, it also holds when conditioning only on partial information (analogously to \eqref{eq:part1}):
\begin{equation}
	P(A_j | \mc{K}_G, \mathbf{Hist}_{j-1} = \text{Hist}_{j-1}, \bS_{j}=S_{j-1}\cup \{j\}, H=H_0) \le \pi^*_j. \label{eq:part1-3}
\end{equation}
Similarly, given a fixed $\text{Hist}_{j-1}$,
\begin{align}
\pi_j = P\big(A_j |\mc{U}_G \cap \mc{D}_G \cap \bigcap_{i \in Y_{j-1}}A_i \cap \bigcap_{i \in N_{j-1}}A_i^c\big) \le  P\big(A_j |\mc{U}_G  \cap \bigcap_{i \in Y_{j-1}}A_i \big)  \le P\big(A_j | \bigcap_{i:h_i \in H_0}A_h \big) = \pi_j^*, \label{eq:part2-3}
\end{align}
analogously to \eqref{eq:part2} in the case $r \ge 4$. We have to work a little harder to obtain a statement similar to \eqref{eq:bdd1}, and for this, we first split up the hypergraph $H=H_0$ into past, present and future, as seen from time $j$. The past is covered by $\mathbf{Hist}_{j-1}=\mathrm{Hist}_{j-1}$, the present is the event $\{h_j\not\in H\} = \{j \notin \bY_j\}$, and we denote the future by $\mathcal H_{\mathrm f}$, i.e.~the event that for all $i>j$ we have $h_i\in H$ if and only if $h_i\in H_0$. So we have
\begin{equation}\label{eq:pastpresentfuture}
P(j\in\bS_j|\mc{K}_G, \mathbf{Hist}_{j-1}=\mathrm{Hist}_{j-1},H=H_0)
=P(j\in\bS_j|\mc{K}_G,\mathbf{Hist}_{j-1}=\mathrm{Hist}_{j-1},h_j\not\in H,\mathcal H_{\mathrm f}).
\end{equation}
Let $\mc{E} = \mc{K}_H \cap \mc{K}_G \cap \{\textbf{Hist}_{j-1}=\text{Hist}_{j-1}, h_j \notin H\}$ and $\mc{E}' = \mc{E} \cap \{j \in \bS_j\}$. Then, noting that $\{H=H_0\} \subseteq \mc{K}_H$, using \eqref{eq:pastpresentfuture} and Bayes' Theorem, we obtain
\[
P(j\in\bS_j|\mc{K}_G, \mathbf{Hist}_{j-1}=\mathrm{Hist}_{j-1},H=H_0) =P(j\in\bS_j|\mcle, H_{\mathrm f})= \frac{P(\mathcal H_{\mathrm f}|\mcle')}
{P(\mathcal H_{\mathrm f}|\mcle)}P(j\in\bS_j|\mcle).
\]
Now, recalling \eqref{eq:defpijprime}, note that the future hyperedge decisions only depend on $j\in\bS_j$ in that this implies $h_j\not\in H$: by construction of the algorithm, the probabilities for future hyperedges decisions for $h_i$, $i >j$, only depend on the information revealed about $H$. But $h_j \notin H$ is already covered by $\mcle$, so we have $P(\mclh_{\mathrm f}|\mcle')=P(\mclh_{\mathrm f}|\mcle)$ and hence
\begin{equation}\label{eq:reduced}
P(j\in\bS_j|\mc{K}_G,\mathbf{Hist}_{j-1}=\mathrm{Hist}_{j-1},H=H_0)
=P(j\in\bS_j| \mc{K}_G, \mc{K}_H, \mathbf{Hist}_{j-1}=\mathrm{Hist}_{j-1},h_j\not\in H).
\end{equation}
Note that $\mc{K}_G, \mc{K}_H, \mathbf{Hist}_{j-1}=\mathrm{Hist}_{j-1}$ is exactly the information on $G$ and $H$ available to us at time $j$ of the algorithm (before deciding $h_j \in H$ and $A_j$). Recall that we assume that $\text{Hist}_{j-1}$ is so that $\pi_j >  \pi_j'$. Then analogously to \eqref{eq:bdd1}, using \eqref{eq:part2-3} we can bound 
 	\begin{equation*}
 		P(j \in \bS_{j} \mid \mc{K}_G, \mc{K}_H, \mathbf{Hist}_{j-1}=\text{Hist}_{j-1}, h_j \notin H) =  \frac{1-\pi_j'/\pi_j}{1-\pi_j'} \le \frac{1-\pi_j'/\pi_j^*}{1-\pi_j'},
 	\end{equation*}
 	because $1-\pi_j'/\pi_j$ is the probability that the coin at step $j$ of the algorithm lands tails (so we add $j$ to $\bS_j$), and $1-\pi_j'$ is the overall probability that we do not add $j$ to $\bY_j$ (that is, $h_j \notin H$), which we've conditioned on. So it follows with \eqref{eq:reduced} that
 	\[
 	P(j\in\bS_j \mid \mc{K}_G,\mathbf{Hist}_{j-1}=\mathrm{Hist}_{j-1},H=H_0) \le \frac{1-\pi_j'/\pi_j^*}{1-\pi_j'}.\]

With \eqref{eq:part1-3}, this gives
 	\[
 	P(A_j \mid \mc{K}_G,\mathbf{Hist}_{j-1}=\mathrm{Hist}_{j-1},H=H_0) \le \frac{\pi_j^*-\pi_j'}{1-\pi_j'}.\]
 	Bounding $\pi_j'$ with Lemma~\ref{lemma:boundforpijprime}~c), this yields
 	 	\begin{equation}\label{eq:bdd1-3}
 		P(A_j \mid \mc{K}_G,\mathbf{Hist}_{j-1}=\mathrm{Hist}_{j-1},H=H_0) \le \frac{\pi_j^*-\pi}{1-\pi}+n^{-3+o(1)}.\end{equation}
 As the bound \eqref{eq:bdd1-3} holds for every possible $\text{Hist}_{j-1}$ with $\pi_j > \pi_j'$, it also holds if we only condition on $\mc{K}_G$, $H=H_0$ and $\pi_j > \pi_j'$, which together with \eqref{eq:equivalentcondition}  gives the result. 
 	\end{proof}

	\subsection{Proof of Lemma~\ref{lemma:lowdegreenoextra}}\label{section:lowdegreeextra}
We will need the following lemma, which was Lemma 13 in \cite{riordan2022random} (for $r \ge 4$) and Lemma 3 in \cite{heckel2020random} (for $r=3$).
\begin{lemma}[\cite{heckel2020random,riordan2022random}] \label{lemma:extraavoidable}
	Fix $r\ge3$, let $H$ be an $r$-uniform hypergraph, and let $G$ be the simple graph obtained by replacing each hyperedge of $H$ by an $r$-clique. Suppose that $G$ contains an $r$-clique $T$ on a set of $r$ vertices and the corresponding hyperedge is not present in $H$.
	\begin{itemize}
	    \item[a)] If $r \ge 4$, then $H$ contains an avoidable configuration.
	    \item[b)] If $r=3$, then $T$ is the middle triangle of a clean $3$-cycle of $H$, or $H$ contains an avoidable configuration.
	\end{itemize}\qed
\end{lemma}

		Now consider the coupled $G$ and $H$ with $p=p_+$ and $\pi=\pi_+$. Condition on $H=H_0$ for some hypergraph $H_0 \notin \mc{B}$. Then in particular, there are at most $(\log n)^{8g(n)}=n^{o(1)}$ low-degree vertices in $H_0$. For $r=3$, no low-degree vertex is contained in the middle triangle of a clean $3$-cycle in $H$ whp by Lemma \ref{lemma:lowdegreenotbad}, as their expected number is of order $\log^3 n$, which we assume in the following. Let $v$ be a low-degree vertex. If $r\ge 4$, then by Lemma \ref{lemma:extrabound} the expected number of `extra' cliques $v$ gets in $G$ is at most
	\begin{align}
	\sum_{\substack{j:v \in h_j\\ h_j \notin H_0}}  \frac{\pi_j^*-\pi_+}{1-\pi_+} \sim \sum_{\substack{j:v \in h_j\\ h_j \notin H_0}}(\pi_j^*-\pi_+) = \sum_{\substack{j:v \in h_j\\ h_j \notin H_0}} (\pi_j^*-p_+^{r \choose 2}) +\sum_{\substack{j:v \in h_j\\ h_j \notin H_0}}(p_+^{r \choose 2} - \pi_+), \label{eq:sums}
	\end{align}
	where we take the sums over hyperedges (or potential cliques) containing $v$ which are not in $H_0$. 	As $\pi_+=p_+^{r \choose 2}(1-n^{-\delta})$, the second sum can be bounded by $n^{-c}$ for some small constant $c>0$.

 Now for $r=3$, we can bound the expected number of `extra' cliques $v$ gets in $G$ with the help of Lemma~\ref{lemma:extrabound3} in the same way, except that we get the extra term $\sum_{\substack{h:v \in h\\ h \notin H_0}}n^{-3+o(1)}$. But this term is of order $n^{-1+o(1)}$, because only of order $n^2$ hyperedges $h$ appear in the sum.

 Thus, for both $r=3$ and $r \ge 4$, it only remains to bound the first sum in \eqref{eq:sums}. This sum is, roughly speaking, the number of extra cliques containing $v$ we expect to get in $G$ if we condition on the cliques corresponding to hyperedges in $H_0$ being present, versus without any conditioning. 
In the following lemma we show that this first sum is also small --- it follows from the lemma and our previous observations that the expected number of extra cliques containing $v$ can be bounded by $n^{-c}$ for some small constant $c>0$. As there are only $n^{o(1)}$ low-degree vertices in $H=H_0$, whp no such vertices exist, which concludes the proof of Lemma~\ref{lemma:lowdegreenoextra}.

	\begin{lemma}\label{lemma:sumbound}
		Suppose that $H_0 \notin \mc{B}_1 \cup \mc{B}_2 
  $, and define $\pi_j^*$ as in Lemmata~\ref{lemma:extrabound} (for $r \ge 4$) and \ref{lemma:extrabound3} (for $r=3$). Let $s = \mathrm{min}\big(\frac 2r, 1- \frac 2r \big) >0$. Let $v$ be an arbitrary vertex, and in the case $r=3$ suppose that $v$ is not an endpoint of the middle triangle of a clean $3$-cycle in $H_0$. Then	\begin{equation}\label{eq:sumbound}
 \sum_{\substack{j:v \in h_j\\ h_j \notin H_0}}(\pi_j^* - p_+^{r \choose 2}) \le n^{-s+o(1)}.
		\end{equation}
	\end{lemma}

	\begin{proof}[Proof of Lemma~\ref{lemma:sumbound}]
		
	Let $G_0$ be the graph obtained by replacing every hyperedge in $H_0$ with an $r$-clique. Then conditioning on $\bigcap_{i: h_i \in H_0} A_i$ is the same thing as conditioning on all the edges of $G_0$ being present.
	
	As $H_0 \notin \mc{B}_1$, the maximum degree of $G_0$ is $n^{o(1)}$. Fixing some vertex $v$, we want to count all possible hyperedges $h_j$ which appear in the sum in \eqref{eq:sumbound}, and bound their contribution to \eqref{eq:sumbound}. Let $C_{h_j}$ be the clique corresponding to the hyperedge $h_j$. For some given $h_j \notin H_0$ so that $v \in h_j$, let the components of the subgraph induced by $h_j$ in $G_0$ have $c_1, \dots, c_t$ vertices respectively, so that $\sum_{i=1}^t c_i = r$. If $c_1 = c_2 = \dots = c_t = 1$, then none of the edges of $C_{h_j}$ are already present as edges in $G_0$, so then $\pi^*_j = p_+^{r \choose 2}$. Hence, such hyperedges $h_j$ contribute $0$ to the sum in \eqref{eq:sumbound}, and in the following we only consider those $h_j$ (and corresponding $C_{h_j}$) where $c_i \ge 2$ for at least one $i$.
		
		Given a fixed sequence $c_1, \dots, c_t$, there are at most
		\begin{equation}\label{eq:cliqueways}
		n^{t-1+o(1)}
		\end{equation}
		ways to pick the clique $C_{h_j}$ (i.e.\ the hyperedge $h_j$): because we already have $v$, we make one choice (from $n$ vertices) for the first vertex from each of the other $t-1$ components, and for each of the $r-t$ remaining vertices there are only $n^{o(1)}$ choices as the maximum degree of $G_0$ is $n^{o(1)}$.
		
		In particular, if $c_1=r$ (so $t=1$), there are only $n^{o(1)}$ choices for $C_{h_j}$. Not all edges in the clique $C_{h_j}$ can be contained in $G_0$ already, otherwise $H_0$ would contain an avoidable configuration --- this follows from Lemma \ref{lemma:extraavoidable}, noting that in the case $r=3$, $v$ is not in the middle triangle of a clean $3$-cycle in $H_0$ --- so $\mc{B}_2$ would hold. So, in the case $c_1=r$ we have
		\[
		\pi^*_j \le p_+ = n^{-2/r+o(1)},\]
		which implies that the overall contribution to \eqref{eq:sumbound} from hyperedges $h_j$ with $c_1=r$ is at most $n^{-2/r+o(1)} \le n^{-s+o(1)}$.
		
		Hence, in the following we only consider sequences $c_1, \dots, c_t$ where $c_i < r$ for all $i$. After conditioning on the edges in $G_0$ being present, a potential clique $C_{h_j}$ is missing (at least) all the edges between different components, so
		\[
		\pi^*_j \le p_+^{{r \choose 2} - \sum_{i=1}^t {c_i \choose 2}}.
		\]
		Together with \eqref{eq:cliqueways}, given a fixed sequence $c_1, \dots, c_t$ (with $c_i<r$ for all $i$ and $c_i \ge 2$ for at least one $i$), the contribution to \eqref{eq:sumbound} from hyperedges $h_j$ (and cliques $C_{h_j}$) corresponding to this sequence is at most
		\begin{align*}
			n^{t-1+o(1)}p_+^{{r \choose 2} - \sum_{i=1}^t {c_i \choose 2}} 
   &= \left( n^{r-1} p_+^{{r \choose 2}} \right)n^{-r+t+ \frac 2r \sum_{i=1}^t {c_i \choose 2} +o(1)} =n^{-\sum_{i=1}^t(c_i-1)+ \frac 1r \sum_{i=1}^t c_i(c_i-1) +o(1)}
   \\
   &= n^{o(1) - \sum_{i=1}^t(c_i-1)(1-c_i/r)}.
		\end{align*}
		Note that all summands in the sum in the exponent are non-negative. Further, there is at least one $i_0$ with $c_{i_0} \ge 2$, and since $c_i <r$ for all $i$, we have $(c_{i_0}-1)(1-c_{i_0}/r) \ge 1-2/r$ (taking the minimum over all possible values $2, \dots, r-1$ for $c_{i_0}$). So, the sum in the exponent above is at least $1-2/r$, and hence we may bound the contribution to \eqref{eq:sumbound} from all hyperedges $h_j$ corresponding to the fixed sequence  $c_1, \dots,c_t$ by
		\[
		n^{-1+2/r+o(1)} \le n^{-s+o(1)}.
		\]
		Since there are only finitely many potential sequences  $c_1, \dots, c_t$, this concludes the proof of the lemma, and thereby of Lemma~\ref{lemma:lowdegreenoextra}.
	\end{proof}
	
\qed

\section{Process coupling - Step 1}\label{section:processcoupling1}

In \S\ref{sec_rc_rcoupling} we revisited and analysed the coupling of $G \sim G(n,p)$ and $H \sim H(n, \pi)$ from Theorem~\ref{theorem:riordancoupling}. Building upon this coupling, we now proceed with the first step in coupling the random graph process with the random hypergraph process.

The aim of \S\ref{section:processcoupling1} is to prove Proposition \ref{prop:coupling1}. Roughly speaking, this states that we may couple the random graph process and the random hypergraph process so that there is \emph{almost} a copy of $H_\HitH$ within the $r$-cliques of $G_\HitG$: for all hyperedges in $H_\HitH$ except those in a set $\mc{E}$ (the \emph{exceptional} hyperedges), there is an $r$-clique in $G_\HitG$ on the same vertex set. Moreover, the exceptional hyperedges in $\mc{E}$ all gain a partner hyperedge within a short time window after time~$\HitH$ (which will allow us to identify and deal with them in the next step of the coupling).

\begin{proposition}\label{prop:coupling1}
We may couple the random graph process $(G_t)_{t=0}^N$ and the random hypergraph process $(H_t)_{t=0}^M$ so that whp the following holds. There is a set of hyperedges $\mc{E} \subset H_\HitH$ 
	 so that
	\begin{enumerate}[label={\alph*)}]
		\item $H_\HitH \setminus \mc{E} \subset \cl(G_\HitG)$, and
		\item for every $h_1 \in \mc{E}$ there is a $h_2 \in H_{\HitH+\left \lfloor g(n)n \right \rfloor}\setminus H_\HitH$ so that $|h_1 \cap h_2| =2$. 
\end{enumerate}
\end{proposition}

The remainder of \S\ref{section:processcoupling1} will be devoted to the proof of Proposition~\ref{prop:coupling1}. In \S\ref{subsection:construction} we start with $G \sim G(n,p_+)$ and $H\sim H(n,\pi_+)$, recalling $p_+, \pi_+$ from \eqref{eq:defcriticalwindow} and \eqref{eq:defcriticalwindow_p}, and coupled as in Theorem~\ref{theorem:riordancoupling}. We then equip $G$ and $H$ with (appropriately coupled) uniform orders $\sigma_G$, $\sigma_H$ of their edges and hyperedges. 

It is clear (e.g.\ by the standard device in Definition \ref{def:mastercoupling}) that $G$ equipped with $\sigma_G$ gives an instance of the first $|G|$ graphs in the random graph process $(G_t)_{t=0}^N$. Similarly, $H$ with $\sigma_H$ can be embedded into the random hypergraph process $(H_t)_{t=0}^M$ as the first $|H|$ hypergraphs. Moreover, note that whp in $G$ every vertex is contained in an $r$-clique, and whp $H$ has minimum degree at least $1$. So whp these initial sequences of the graph and hypergraph processes capture $T_G$ and $T_H$. Thus, the coupling of $\sigma_G, \sigma_H$ will induce a coupling of $G_\HitG$ and $H_\HitH$. Furthermore, it follows from \eqref{eq:mastercoupling} that whp $|H|-\HitH \le (1+o(1))(\pi_+-\pi_-){n \choose r} \sim 2 g(n)n/r < \left \lfloor g(n)n \right \rfloor$. So to prove Proposition~\ref{prop:coupling1}, it is sufficient to prove the following statement.

\begin{proposition}\label{prop:coupling2}
We may construct a coupling $(G, \sigma_G, H, \sigma_H)$ so that all of the following hold.
	\begin{enumerate}[label={\alph*)}]
		\item We have $G \sim G(n,p_+)$, and $\sigma_G$ is a uniformly random order of the edges of $G$.
		\item We have $H \sim H(n,\pi_+)$, and $\sigma_H$ is a uniformly random order of the hyperedges of $H$.
		\item \label{item:prope} Let $e_{\sigma_G(1)}, \dots, e_{\sigma_G(|G|)}$ be the edge order given by $\sigma_G$ and $h_{\sigma_H(1)}, \dots, h_{\sigma_H(|H|)}$ the hyperedge order given by $\sigma_H$. For $t \in \mathbb{N}_0$, set
		\[
		G_t^* = \{e_{\sigma_G(1)}, \dots, e_{\sigma_G(t)}\} \text{ and } H_t^* = \{h_{\sigma_H(1)}, \dots, h_{\sigma_H(t)}\} 
		\]
and define
\begin{align*}
	\HitGS &= \min\{t: \text{ every vertex of $G^*_t$ is in an $r$-clique }\},\\
\HitHS  &= \min\{t: \text{ $H^*_t$ has minimum degree $1$}\};
	\end{align*}
	then whp $\HitGS, \HitHS$ are finite. Furthermore, whp there is a set $\mc{E} \subset H^*_{\HitHS}$ such that
		\begin{equation} \label{eq:e1}
H^*_{\HitHS} \setminus \mc{E}	\subset	\cl(G^*_{\HitGS}),
		\end{equation}
		and so that for every $h_1 \in \mc{E}$ there is a $h_2 \in H \setminus H^*_{\HitHS}$ such that
		\[
		|h_1 \cap h_2|=2.
		\]
			\end{enumerate}
\end{proposition}

\subsection{Proof of Proposition \ref{prop:coupling2}}

We start with the coupled $G, H$ as in Theorem~\ref{theorem:riordancoupling}, with $p=p_+$ and $\pi=\pi_+$. In \S~\ref{subsection:construction}, we define coupled uniform orders $\sigma_G$ of the edges of $G$ and $\sigma_H$ of the hyperedges of $H$. Thereafter, we show that item \ref{item:prope} of Proposition~\ref{prop:coupling2} holds.

The idea of the coupling is the following. If we put a uniform order on the edges of $G$, then for any edge-disjoint set of $r$-cliques, the induced order in which they emerge in the random graph process is uniform by symmetry. The only `obstacle', therefore, are the pairs of partner hyperedges in $H$, and the corresponding edge-overlapping cliques in $G$, which share two vertices (whp there are no higher overlaps). For each pair of partner hyperedges, we put in a `dummy edge', and choose a uniform order of the edges and dummy edges. For the hyperedges (cliques), we take the induced order where each partner in a pair of partner hyperedges uses a distinct edge (or dummy edge), which by symmetry is uniform. This will be made more precise below.

\subsubsection{Auxiliary time and construction of $\sigma_G$ and $\sigma_H$}\label{subsection:construction}

In the following, we will want to compare the times at which events happen in the graph process and in the hypergraph process --- for example, to check that the graph and hypergraph hitting times match whp in some natural way, or happen `at the same time'. For a convenient notion of what this means, we rescale time appropriately, introducing an \emph{auxiliary time} in the interval $[0,1]$ (this will be made more precise below). Every edge $e \in G$ and hyperedge $h \in H$ will be assigned a time $\tau(e)$ and $\tau(h)$ in the interval $[0,1]$, so that $\sigma_G$ is the edge order given by ordering the edges in ascending order according to $\tau(e)$, and $\sigma_H$ is the hyperedge order given by ordering the hyperedges in ascending order according to $\tau(h)$. This auxiliary time will allow us to compare when events happen in the graph and hypergraph process.

We now describe the construction of $\sigma_H$, $\sigma_G$. Note that, whp, the coupling of $H$ and $G$ does not fail and we have $H \notin \mc{B}$. In particular, whp we have $H \subset \cl(G)$, no two hyperedges of $H$ overlap in three or more vertices, no hyperedge has more than one partner hyperedge, and there are at most $\log^3 n$ pairs of partner hyperedges in total (see Remark~\ref{remark:vertexdisjoint}). As these properties hold whp, it suffices to specify the construction of $\sigma_G$ and $\sigma_H$ if they hold, which we assume from now on.

So denote the pairs of partner hyperedges of $H$ in some arbitrary order by  $(u_1, v_1), \dots, (u_k, v_k)$, where $k \le \log^3 n $. 
Set
\begin{align*}
S_1 = \{u_1, \dots, u_k\} \quad  \text{ and } \quad S_2 =\{v_1, \dots, v_k\}.
\end{align*}
 For each $h \in H$, let $E(h)\subset G$ be the edge set of the clique in $G$ on the same vertex set. For $i=1, \dots, k$, let $e_i^*$ be the shared edge of the partner hyperedges $(u_i, v_i)$, that is, $\{e_i^*\} = E(u_i) \cap E(v_i)$.

Now to construct the uniform orders, let
\[
 \{\xi_e : e \in G\} \cup \{\xi'_i: 1 \le i \le k\} 
\]
be uniform random variables on $[0,1]$, independent from each other and everything else. Almost surely, all their values are all distinct, which we assume from now on.

For $e \in G$, let
\[\tau(e) = \xi_e  \in [0,1]\]
be the auxiliary time of $e$, and order the edges in $G$ by the value of  $\tau(e)$ in increasing order. This defines the order $\sigma_G$, which clearly is uniform by symmetry. 

For $h \in H \setminus S_2$, let
\[
\tau(h) = \max \{ \xi_e: e \in E(h)\} 
\]
and for $h=v_i \in S_2$, let
\[
\tau(v_i) = \max \Big( \{ \xi_e: e \in  E(v_i) \setminus e_i^*\} \cup \{\xi'_i\} \Big).
\]
Now order the hyperedges of $H$ by the size of the auxiliary times $\tau(h)$ in increasing order, defining the order $\sigma_H$.  As the times $\tau(h)$ each depend on disjoint sets of ${r \choose 2}$ i.i.d.\ random variables $\xi_e$, by symmetry $\sigma_H$ is uniform.

\subsubsection{Induced clique order versus hyperedge order} \label{section:induced_clique}

The edge order in the graph process induces a (partial) order of the clique edge sets $\{E(h): h \in H\}$, given by the auxiliary times
\[
\tau(E(h)) := \max \{\tau(e): e \in E(h) \} \in[0,1].
\]
In the case of partner hyperedges $u_i, v_i$, the corresponding cliques may appear at the same time $\tau(E(u_i))=\tau(E(v_i))$ --- this happens if the shared edge $e_i^*$ is the last to appear, i.e.\ if $\tau(e_i^*) = \max \{\tau(e): e \in E(u_i) \cup E(v_i) \}$.

It is clear from the definition that for all $h \in H \setminus S_2$, we have
\begin{equation} \label{eq:orderprop}
\tau(h) = \tau(E(h)),
\end{equation}
so $h$ and its corresponding clique appear at the same (auxiliary) time. However, for $h \in S_2$, $\tau(h)$  and $\tau(E(h))$ do not always agree, and so the order of the hyperedges of $H$ differs slightly from the induced clique order in $G$. 

The crucial property of our construction is that we maintain some control over when pairs of partner hyperedges appear; we formulate the following key observation as a lemma.

\begin{lemma}\label{lemma:couplingpartners}
For all $i=1, \dots, k$, 
\[
\max(\tau(E(u_i)) , \tau (E(v_i)) )\le \max (\tau(u_i), \tau(v_i)).
\]
That is, if both $u_i$ and $v_i$ are present in the hypergraph process, then both of the corresponding cliques with edge sets $E(u_i), E(v_i)$ are present in the coupled graph process as well (in the auxiliary time correspondence given by $\tau(\cdot)$). \qed
\end{lemma}

\subsubsection{Analysis of the coupled hitting times} \label{subsection:hittingtimes}

In this section we show that, in auxiliary time, the hitting times of the graph process and the hypergraph process match whp. First of all, note that by our choice of $p=p_+$ and $\pi=\pi_+$, whp the hitting times $\HitGS, \HitHS$ as defined in \ref{item:prope} of Proposition~\ref{prop:coupling2} are finite (because $p_+, \pi_+$ are at the end of the `critical window', see \S\ref{section:criticalwindow}). Now for $t \in [0,1]$, let 
\begin{equation} \label{eq:defHt}
H(t) = \{h \in H: \tau(h) \le t\}
\end{equation}
and 
\begin{equation} \label{eq:defGt}
G(t) = \{e \in G: \tau(e) \le t\},
\end{equation}
and define
\begin{align}
t_G &= \inf \{t \in[0,1]: \text{in $G(t)$ every vertex is contained in at least one $r$-clique} \}, \label{eq:deftG} \\
t_H &= \inf\{t \in [0,1]: \text{$H(t)$ has minimum degree at least $1$} \}. \label{eq:deftH}
\end{align}
Then whp $t_G$ and $t_H$ are finite, and in fact $G(t_G)$ and $H(t_H)$ are exactly $G^*_{\HitGS}$ and $H^*_{\HitHS}$.

\begin{lemma}\label{lemma:hittingtimesmatch}
	Whp, $t_G = t_H$.
\end{lemma}

\begin{proof}[Proof of Lemma \ref{lemma:hittingtimesmatch}]
Our overall strategy is as follows: We will first define an auxiliary time $t_-\in[0,1]$ so that $H_-:=H(t_-)$ has essentially the distribution $H(n, \pi_-)$.
Recall that $\pi_-$ was defined in \eqref{eq:defcriticalwindow} as the beginning of the `critical window' for $H$ having minimum degree at least $1$. In particular whp $H(n, \pi_-)$ and therefore $H_-$ has at least one isolated vertex and so $t_- < t_G$.

We then consider the set $\mathcal I \neq \emptyset$ of isolated vertices of $H_-$, and prove that, whp, (a) $\mc{I}$ is exactly the set of vertices not contained in any $r$-clique in $G_-:= G(t_-)$, and (b) all vertices in $\mathcal{I}$ get their first hyperedge (in the hypergraph process) and their first $r$-clique (in the graph process) at the same auxiliary time in $[0,1]$. This implies $t_G=t_H$.

Turning to the details, let $F(y)= P(Y \le y)$ be the distribution function of the random variable $Y=\max_{1\le i \le {r \choose 2} }U_i$, where the random variables $U_i$, $1\le i \le {r \choose 2}$, are independent and uniformly distributed on $[0,1]$. Set
\[
t_-=F^{-1}(\pi_-/\pi_+), \quad \text{and} \quad H_- = H(t_-).
\]
Then by the construction of auxiliary time and the graphs $H(t)$, $t \in [0,1]$, given some $H=H(1) \notin \mc B$, any hyperedge $h \in H$ is also in $H_{-}$ with probability exactly $\pi_-/\pi_+$. So, since $H \sim H(n, \pi_+)$ and whp $H \notin \mc{B}$, essentially $H_{-}$ is distributed as $H(n, \pi_-)$. More specifically, $H_{-}$ can be coupled with the hypergraph $H_-' \sim H(n, \pi_-)$ so that 
\begin{equation} \label{eq:hminuscoupling}
H_- = H_-' \text{ whp.}
\end{equation}
In particular, whp $H_-$ contains isolated vertices and so whp $t_- < t_G$. Let
\[
\mc{I} = \{v \mid v \text{ is an isolated vertex in }H_-\},
\]
then whp $\mc{I} \neq \emptyset$.

Consider the following two lemmata. Their proofs are routine applications of the first moment method and will be given in the appendix. Recall that we called a vertex in a hypergraph low-degree if $d(v) \le 7 g(n)$  (where $g(n)$ is the function from \eqref{eq:defcriticalwindow} which we fix throughout). Let $S$ be the set of hyperedges of $H$ with a partner hyperedge (another hyperedge overlapping in exactly two vertices). So in the notation of \S \ref{subsection:construction}, we have $S=S_1 \cup S_2$. 

\begin{lemma} \label{lemma:isolatedlowdegree}
	Whp all isolated vertices in $H_-$ are low-degree vertices in $H$.
\end{lemma} 
\begin{lemma} \label{lemma:isonotbad}
	Whp all non-isolated vertices in $H_-$ are incident with at least one hyperedge from $H \setminus S$ in $H_-$.
\end{lemma}

Suppose that the whp-statements of Lemmata \ref{lemma:lowdegreenoextra}, \ref{lemma:isolatedlowdegree} and \ref{lemma:isonotbad} all hold, and that $H \notin \mc{B}$. Further, note that it follows from Lemma~\ref{lemma:lowdegreenotbad} that whp no low-degree vertex of $H$ is incident with any hyperedges from $S$, and in the case $r=3$ with any clean $3$-cycles, since the expected numbers of pairs of partner hyperedges and of clean $3$-cycles are $\Theta(\log^2 n)$ and $\Theta(\log^3 n)$, respectively --- so suppose that these whp-statements also hold. 

Then the following is true for any $v \in \mc{I}$. By Lemma \ref{lemma:isolatedlowdegree}, $v$ is a low-degree vertex in $H$. By Lemma~\ref{lemma:lowdegreenoextra}, in the coupled graph $G$, $v$ is not included in any extra cliques, so any $r$-cliques containing $v$ are represented by hyperedges in $H$. Furthermore, by Lemma \ref{lemma:lowdegreenotbad}, $v$ is not contained in any hyperedges which have a partner hyperedge. Recall that in the construction in \S\ref{subsection:construction}, any hyperedge of $H$ which does \emph{not} have a partner hyperedge is assigned the same auxiliary time as the corresponding $r$-clique in $G$ (see \eqref{eq:orderprop}). So all hyperedges incident with $v$ have the same auxiliary times as the corresponding $r$-cliques of $G$.  In particular, $v$ is not contained in any $r$-cliques in $G_-=G(t_-)$.

So none of the vertices in $\mc{I}$ are contained in any $r$-cliques in $G_-$; and for any hyperedge $h \in H$ which meets $\mc{I}$, $\tau(h)=\tau(E(h))$. Furthermore, $G_-$ does noes not have any other vertices $v \notin \mc I$ not contained in any $r$-cliques: By the whp-statement of Lemma \ref{lemma:isonotbad}, any $v \notin \mc{I}$ is incident with a hyperedge $h\in H_-\setminus S$, which implies $\tau(h)=\tau(E(h))$, and therefore $v$ is contained in the clique $E(h)$ in $G_-$.

In summary, $\mc{I}$ is exactly the set of vertices not contained in any $r$-clique in $G_-$, and for any hyperedge $h \in H$ meeting $\mc{I}$, $\tau(h) = \tau(E(h))$. This implies $t_G=t_H$.
\end{proof}

\subsubsection{Finishing the proof of Proposition \ref{prop:coupling2}} 
Proposition \ref{prop:coupling2} now follows easily from the construction of $\sigma_G$ and $\sigma_H$ in \S\ref{subsection:construction} and Lemma~\ref{lemma:hittingtimesmatch}: Let
\begin{equation}\label{eq:defE}
\mathcal{E}=H^*_{\HitHS} \setminus \mathrm{cl}(G^*_{\HitGS}) = H(t_H) \setminus \mathrm{cl} (G(t_G)), 
\end{equation}
then by definition $\mathcal{E}$ satisfies \eqref{eq:e1}. By Lemma~\ref{lemma:hittingtimesmatch}, whp $t_H=t_G$, which we assume from now on, and in particular $\mathcal{E}=H(t_H) \setminus \mathrm{cl} (G(t_H))$. By \eqref{eq:orderprop}, for all $h \in H(t_H) \setminus S_2$ we have
\[
\tau(h) = \tau(E(h)),
\]
so for any such $h$ the corresponding clique with edge set $E(h)$ is present in $G(t_H)$. Hence, we have $\mathcal{E} \subset S_2$.  

So suppose that the hyperedge $v_i \in S_2$ is in $\mathcal{E}$, and let $u_i \in S_1$ be the partner hyperedge of $v_i$. To prove \ref{item:prope}, we need to show that $u_i \in H \setminus H(t_H)$. So suppose that is not the case, that is, suppose that $u_i \in H(t_H)$. But then both $u_i$ and $v_i$ are in $H(t_H)$, which implies
\[
\max(\tau(u_i), \tau(v_i)) \le t_H = t_G.
\]
It follows by Lemma~\ref{lemma:couplingpartners} that
\[
\tau(E(v_i)) \le \max\big(\tau(E(u_i)), \tau(E(v_i))\big) \le \max(\tau(u_i), \tau(v_i)) \le t_G,
\]
which implies $v_i \in \mathrm{cl}(G(t_G))$. But $v_i \in \mc{E}$, so this contradicts the definition \eqref{eq:defE} of $\mathcal{E}$. It follows that $u_i \in H \setminus H(t_H)$, giving \ref{item:prope}. This concludes the proof of Proposition~\ref{prop:coupling2}. \qed

\section{Process coupling - Step 2}\label{section:step2}

The aim of this section is to show that, whp, we can embed the set $\mc{E}$ of `exceptional' hyperedges from Proposition~\ref{prop:coupling1} into a random set $\mc{R}$ which includes every $h \in H_\HitH$ independently with a small probability.

\begin{proposition}\label{prop:couplingrandomset}
 
We may couple the random $r$-uniform hypergraph process $(H_t)_{t=0}^M$ and a set $\mc{R}\subset {[n] \choose r}$ of hyperedges so that both of the following properties hold.
\begin{enumerate}[label={\alph*)}]
	\item \label{randomset1} We have $\mc{R} \subseteq H_\HitH$, and (given only $H_\HitH$) each hyperedge $h \in H_\HitH$ is included in $\mc{R}$  independently with probability
	\begin{equation}\label{eq:pirasymptotics}
	\pi_\mc{R} = \frac{10r^4 g(n)}{n}.
	\end{equation}
	\item Let $\mc{F} \subset H_\HitH$ be the set of hyperedges in $H_\HitH$ which have a partner hyperedge in
	\[H_{\HitH+\left \lfloor g(n)n \right \rfloor} \setminus H_{\HitH}. \]
	Then, whp,
	\begin{equation}\label{eq:FsubsetR}
	\mc{F} \subset \mc{R}.
	\end{equation}
	\end{enumerate}
\end{proposition}

\begin{proof}
We start by constructing the random hypergraph process $H_0, H_1, \dots, H_\HitH$ up to the hitting time $\HitH$ as usual.
For each $h \in H_\HitH $, there are at  most
\[
{r \choose 2}{n-r \choose r-2}  
\]
potential hyperedges which could be a partner of $h$ and which are not present in $H_\HitH$. 
Of those potential partner hyperedges, there may be some

which would, at the same time, be a partner to another hyperedge $h' \in H_\HitH$.

Let $\mc{X}_h$ be the set of potential \emph{exclusive} partner hyperedges which $h$ would not share with any other hyperedge in $H_\HitH$, and set
\begin{equation}\label{eq:boundxh}x_h = |\mc X_h| \le {r \choose 2}{n-r \choose r-2}\le\frac{r(r-1)}{2(r-2)!}n^{r-2}. \end{equation} 

We first define an auxiliary random hyperedge set $I$, and use this to construct the remaining hypergraph process and $\mc R$. So let $I$ be a set of hyperedges in which we include every $h \notin H_\HitH$ independently with probability
\[
\pi_I = \frac{10 g(n) r!}{n^{r-1}} .
\]
We realise the rest of the random hypergraph process by putting a uniform order on the hyperedges in $I$ and including them in the random hypergraph process in that order (yielding $H_{\HitH+1}, \dots, H_{\HitH+|I|}$), and then including the remaining hyperedges in a uniform order. Then by symmetry the overall hyperedge order is uniform.

Now we turn to the construction of $\mc R$. Let
\begin{equation}\label{eq:defRprime}
\mc R' = \{h \in H_\HitH 
\mid \mc X_h \cap I \neq \emptyset\}
\end{equation}
be the set of hyperedges with an exclusive partner hyperedge in $I$.
Note that by \eqref{eq:boundxh}, every $h\in H_\HitH$ is included in $\mc R'$ independently with probability exactly
	\begin{equation}\label{eq:pihasymptotics}
	\pi_h = 1- (1-\pi_I)^{x_h} \sim \pi_I x_h \le 5 r^2 (r-1)^2 g(n) /n,
	\end{equation}
independently of the other hyperedges. Note that  $\pi_h \le \pi_\mc{R}$ by \eqref{eq:pirasymptotics}, so to construct the set $\mc R$ from $\mc R'$ we only need to increase the probability of including each $h$ slightly. We do this in a standard way: for each $h \in H_\HitH$:
\begin{itemize}
\item If $h \in \mc R'$ (which happens with probability $\pi_h$), include $h$ in $\mc R$.
\item If $h \notin \mc R'$, include $h$ in $\mc R$ with probability $\pi_h'$ defined by
		\begin{equation*}
        \pi'_h=\frac{\pi_\mc{R}-\pi_h}{1-\pi_h}
		\end{equation*}
		independently of everything else.
\end{itemize}
Then each $h \in H_\HitH$ is included in $\mc R$ independently with probability exactly $\pi_\mc{R}$.

It only remains to show that whp \eqref{eq:FsubsetR} holds. First note that, since whp $\HitH = O(n \log n)$, $|I| \sim \mathrm{Bin}({n \choose r}-\HitH, \pi_I)$ has expectation $(1+o(1))10g(n)n \rightarrow \infty$, and so whp $|I| \ge ng(n)$. Thus by construction of the hypergraph process, whp
\begin{equation}\label{eq:diffinI}
H_{\HitH+\left \lfloor g(n)n \right \rfloor} \setminus H_\HitH \subseteq I.
\end{equation}
Further, whp $H_\HitH \cup I = H_{\HitH+|I|}$ contains no avoidable configurations, and so in particular no hyperedge with two (or more) partner hyperedges.\footnote{In more detail: Note that because whp $\HitH +|I| = O(n \log n + \pi_I n^{r-1})=O(n \log n)$, by the standard coupling in~\S\ref{section:criticalwindow}, $H_{\HitH+|I|}$ can whp be embedded into $H' \sim H(n, \pi')$ for some $\pi' = n^{1-r+o(1)}$, and by Lemma~\ref{lemma:Hnotbad}, $H'$ contains no avoidable configuration.} So whp $I$ does not contain any non-exclusive partner hyperedges, that is, whp
\begin{equation}\label{eq:IsubsetX}
\{h'\in I:h\in H_\HitH,|h\cap h'|\ge 2\}\subset \bigcup_{h \in H_\HitH} \mc{X}_h.
\end{equation}
Now suppose that \eqref{eq:diffinI} and \eqref{eq:IsubsetX} hold. If $h \in \mc F$, then $h$ has a partner hyperedge $h'$ in 
\[
\{h'\in H_{\HitH+\left \lfloor g(n)n \right \rfloor} \setminus H_\HitH:|h'\cap h|\ge 2\} \subseteq 
\{h'\in I:h\in H_\HitH,|h\cap h'|\ge 2\}
\subseteq \bigcup_{h \in H_\HitH} \mc{X}_h. \]
This partner hyperedge $h'$ cannot be in some $\mc{X}_{h''}$ for $h'' \neq h$, because then it would not be an exclusive partner hyperedge to $h''$. So it follows that $h' \in \mc{X}_{h}$, and so $h' \in \mc{X}_h \cap I$. But then \eqref{eq:defRprime} implies that $h \in \mc R'$. Since $h \in \mc F$ was arbitrary and $\mc R' \subseteq \mc R$, it follows that $\mc F \subseteq \mc R$.

\end{proof}

\section{Process coupling - Step 3}\label{section:step3}
As the final piece of the puzzle, we show that after removing every hyperedge from $H_\HitH$ independently with a small probability, whp we still have an instance of the stopped random hypergraph process.

\begin{proposition}\label{prop:couplingnewprocess}
Let $H_\HitH$ be the stopped random hypergraph process, and let $\mc{R}\subset H_\HitH$ be a subset of hyperedges where we include every $h \in H_\HitH$ independently with probability
		\[
		\pi_\mc{R} = \frac{10r^4 g(n)}{n}.
		\]
We may couple $H_\HitH$ and $\mc{R}$ with another instance $H'_{\HitHP}$ of the stopped random hypergraph process so that, whp, 
\[
H_\HitH \setminus \mc{R} = H'_{\HitHP}.
\]
\end{proposition}

\begin{proof} 
Let $H_\HitH$ be given together with the whole random hypergraph process $(H_t)_{t=0}^M$, with the hyperedges appearing in the uniform order
	 \[
	 h_1, h_2, \dots, h_M.
	 \]
	Now remove every hyperedge from this sequence independently with probability $\pi_\mc{R}$, resulting in a \emph{thinned sequence}
	\[
	h_1', h_2', \dots, h_{M-k}',
	\]
	where $k \sim \mathrm{Bin}(M, \pi_\mc{R})$ is the number of removed hyperedges. Note that every potential hyperedge $h$ appears with probability $1-\pi_\mc{R} =1-o(1)$ somewhere in this thinned sequence, independently, and the order of the hyperedges in the thinned sequence is still uniform. 
	Conditional on $k$, this describes an instance of the first $M-k$ hypergraphs in the random hypergraph sequence.
			
	Considering $H_\HitH$, we removed a set $\mc{R}$ of hyperedges independently with probability $\pi_\mc{R}$ each, yielding some hypergraph $H'_{t_0}$ in the thinned sequence. It remains to show that whp $t_0=\HitHP$. We certainly have $t_0 \le \HitHP$, because we only removed hyperedges, so we cannot have minimum degree 1 any earlier than before. So we need to show that $\HitHP \le t_0$ whp, that is, whp $H'_{t_0}$ has no isolated vertices.
	
	This follows easily from the  following observations:
	\begin{enumerate}
\item Let $T_1, T_2$ 
be independent of everything else, with $T_1 \sim \mathrm{Bin}\Big(M, \pi_- \Big)$ and $T_2 \sim \mathrm{Bin}\Big(M, \pi_+ \Big)$. Then $H_{T_1} \sim H(n, \pi_-)$ and $H_{T_2} \sim H(n, \pi_+)$. Furthermore whp $T_1 \le \HitH \le T_2$, and so $H_{T_1} \subseteq H_\HitH \subseteq H_{T_2}$ whp.
This is the well-established critical window, see \eqref{eq:defcriticalwindow}.

\item Whp, no vertex is incident in $H_\HitH$ with more than one hyperedge in $\mc{R}$.

To see this, consider the hypergraph $H_{T_2}$. The expected number of pairs of hyperedges overlapping in at least one vertex which are present in $H_{T_2}$ but later removed with probability $\pi_\mc{R}$ is of order $n^{2r-1} \pi_+^2 \pi_\mc{R}^2=O(\log(n)^2 g(n)^2 / n) =o(1)$. So whp there are no such two hyperedges in $H_{T_2}$, nor in $H_\HitH \subseteq H_{T_2}$.

\item Whp, no vertex is incident in $H_\HitH$ with one hyperedge in $\mc{R}$, and no other hyperedges in $H_{T_1}$.

If $H_{T_1} \subseteq H_\HitH \subseteq H_{T_2}$, which holds whp, such a vertex would have to be incident with a hyperedge $h \in H_{T_2}$ which is removed in the thinned process, and in $H_{T_1}$ it would be incident with no other hyperedges. The probability of that is of order 

\[O\left( n^{r-1} \pi_+ \pi_\mc{R} (1-\pi_-)^{{n-1 \choose r-1}-1} \right)=  O \left( \log n g(n)e^{g(n)}/n^2\right)=o(1/n).
\]
So whp there is no such vertex.

		\end{enumerate}
		As $H_\HitH$ has no isolated vertices, it follows from the second and third observation, we find that whp $H_\HitH \setminus \mc{R}=H'_{t_0}$ also has no isolated vertices. So whp $\HitHP \le t_0$ as required. 
	\end{proof}
\section{Proof of Theorem \ref{theorem:finalcoupling}}\label{section:puteverythingtogether}
We now have all our ducks in a row. Combining Propositions \ref{prop:coupling1}, \ref{prop:couplingrandomset} and \ref{prop:couplingnewprocess}, we obtain a chain of couplings that whp embeds the stopped hypergraph process into the cliques of the stopped graph process:
	\[
H'_{\HitHP}\quad \quad \stackrel{\text{whp}}{=} \quad \quad H_{\HitH} \setminus \mc{R} \quad \quad \stackrel{\text{whp}}{\subseteq} \quad \quad H_{\HitH} \setminus \mc{F} \quad \quad \stackrel{\text{whp}}{\subseteq} \quad \quad \mathrm{cl}(G_\HitH)
\]
In more detail:
First of all, note that we may combine different couplings via the Gluing Lemma (which is trivial in this finite setting\footnote{Given couplings of $X$ and $Y$ and of $Y$ and $Z$, i.e., desired distributions for $(X,Y)$ and for $(Y,Z)$, we construct $(X,Y,Z)$ by starting with $Y$ and, given the value of $Y$, taking the appropriate conditional distributions for $X$ and for $Z$ -- for example with conditional independence.}). We start with a stopped hypergraph process $H'_{\HitHP}$, which we couple according to Proposition \ref{prop:couplingnewprocess} with a stopped hypergraph process $H_\HitH$ and a set $\mc{R} \subset H_\HitH$ where every hyperedge is included independently with probability $\pi_\mc{R}$, so that whp 
\[H'_{\HitHP} = H_\HitH \setminus \mc{R}.\]
Then, via Proposition \ref{prop:couplingrandomset}, we embed $H_\HitH$ into a full hypergraph process $(H_t)_{t=0}^M$ so that whp $\mc{F} \subset \mc{R}$, where $\mc{F}$ is the set of all $h \in H_\HitH$ with a partner hyperedge in $H_{\HitH+ \left \lfloor n g(n) \right \rfloor} \setminus H_\HitH$. It follows that whp
\[H_\HitH \setminus \mc{R} \subseteq H_\HitH \setminus \mc{F}.\]
Finally, couple $(H_t)_{t=0}^M$ with the random graph process $(G_t)_{t=0}^N$ via Proposition \ref{prop:coupling1}. By part b) of Proposition \ref{prop:coupling1}, we have $\mc{E} \subset \mc{F}$, and so whp 
\[
H_\HitH \setminus \mc{F} \subseteq H_\HitH \setminus \mc{E} \subset \mathrm{cl}(G_{T_G}).
\]
Following through the chain of couplings, we have coupled $H'_{\HitHP}$ and $G_\HitG$ so that, whp, 
\[
H'_{\HitHP} \subseteq  \mathrm{cl}(G_{T_G}).
\]
\qed

	\section{The hitting time of $K_r^{(s)}$-factors in random $s$-uniform hypergraphs}\label{section:hypergraphfactor}

In the following, for $r>s\ge 3$, let $K_r^{(s)}$ denote the complete $s$-uniform hypergraph on $r$ vertices. By adapting the proof of Theorem~\ref{theorem:finalcoupling}, it is straightforward to prove a corresponding result for $K_r^{(s)}$-factors. To emphasize the analogy with Theorem~\ref{theorem:finalcoupling}, let $(G_t)_{t=1}^{N_s} = (H_t^s)_{t=1}^{N_s}$ and 
\[
\HitG = \mathrm{min}\{ t : \text{ every vertex in $G_t$ is contained in at least one $K_r^{(s)}$}\}
\]
be the hitting time of a $K_r^{(s)}$-cover. Finally, for an $s$-uniform hypergraph $G$, let $\cl(G)$ be the set of vertex sets from $\binom{V}{r}$ which span the copies of $K_r^{(s)}$ in $G$, such that $\cl(G)$ is again an $r$-uniform hypergraph in the usual sense. Then, in fact, the following holds true:

\begin{theorem} \label{theorem:finalcouplinghyper}
	Let $r > s \ge 3$. We may couple the stopped random $r$-uniform hypergraph process $H_\HitH$ and the stopped random $s$-uniform hypergraph process $G_\HitG$ so that, whp,
	\[
	H_\HitH \subseteq \cl(G_\HitG).
	\]
	That is, whp, for every hyperedge in $H_\HitH$ there is copy of $K_r^{(s)}$ in $G_\HitG$ on the same vertex set. 
	\end{theorem}

The combination of Theorems~\ref{theorem:kahnhitting} and \ref{theorem:finalcouplinghyper} then immediately implies the following:

\begin{corollary}
    \label{corollary:hypergraphhitting}
	Let $r > s \ge 3$ and $n \in r \mathbb{Z}_+$, then whp $G_\HitG$ has a $K_r^{(s)}$-factor.
\end{corollary}

 The proof of Theorem~\ref{theorem:finalcouplinghyper} proceeds exactly along the lines of the proof of our main theorem, Theorem \ref{theorem:finalcoupling}, and we therefore only point out the necessary modifications in the following.  Moreover, it is actually slimmer than that of Theorem~\ref{theorem:finalcoupling}, due to the major simplification that whp, no two hyperedges of $H_r(n,\pi_+)$ overlap in three or more vertices (see Remark~\ref{remark:vertexdisjoint}). Therefore, there are no ``partner hyperedges'' whose corresponding $K_r^{(s)}$'s would not appear independently, and we can simply take uniform orders without dummy edges in the process coupling. Correspondingly, Theorem~\ref{theorem:finalcouplinghyper} will be established once we have proved the analogue of Proposition~\ref{prop:coupling2}.

\subsection{Terminology and critical window}
Throughout the proof of Theorem~\ref{theorem:finalcouplinghyper}, to illustrate the analogy with the case $s=2$, we will often use the word ``clique'' for the hypergraph $K_r^{(s)}$ or $s$-edge for a hyperedge of cardinality $s$.

We start from the following extension of Riordan's coupling Theorem~\ref{theorem:riordancoupling}:

\begin{theorem}[{\cite[Thm.~5]{riordan2022random}}]\label{theorem:riordancoupling_extension}
	Let $r>s$ with $s \ge 2$ and $r\ge 4$. There are constants $\epsilon(s,r),\delta(s,r)>0$ such that, for any $p=p(n) \le n^{-(r-1)/\binom{r}{s} + \epsilon}$, the following holds. Letting $\pi= (1-n^{-\delta})p^{r \choose s}$, we may couple the random s-uniform hypergraph $G=H_s(n,p)$ with the random hypergraph $H=H_r(n,\pi)$ 
	so that, whp, for every hyperedge in $H$ there is a copy of $K_r^{(s)}$ in $G$ on the same vertex set.\footnote{Again, in \cite{riordan2022random}, Theorem~\ref{theorem:riordancoupling_extension} was given with an unspecified $o(1)$-term in place of $n^{-\delta}$. As in the case $s=2$, the formulation above follows from a closer inspection of the proof of \cite[Thm.~5]{riordan2022random}.}
\end{theorem}

The coupling of Theorem~\ref{theorem:riordancoupling_extension} is identical to the coupling presented in Section~\ref{section:riordanalgo}, replacing every ``$r$-clique'' by ``copy of $K_r^{(s)}$'' along the way.

For $\pi_-, \pi_+$ as in (\ref{eq:defcriticalwindow}) and the constant $\delta$ from Theorem~\ref{theorem:riordancoupling_extension}, 
let 
\begin{equation} \label{eq:defcriticalwindow_hyper}
p_{\pm} = p_{\pm}(s) = (\pi_{\pm}/(1-n^{-\delta}))^{1/{r \choose s}}.
\end{equation}

Note that for $n$ large enough, we have $p_+ \leq n^{-(r-1)/\binom{r}{s}+\varepsilon}$, so we may apply Theorem~\ref{theorem:riordancoupling_extension} with $p=p_+$ and $\pi=\pi_+$ later on.

\subsection{Coupling of {$H_s(n,p)$} and {$H_r(n,\pi)$}}

The process coupling from Theorem~\ref{theorem:finalcouplinghyper} proceeds exactly as before: Given an elaborate coupling of $G \sim H_s(n, p_+)$ and $H \sim H_r(n,\pi_+)$, we equip $G$ and $H$ with appropriately coupled uniform orders $\sigma_G, \sigma_H$ of their respective hyperedges. This section explains the main features of the coupling $(G,H)$ that are needed for the process version.

The goal of the current section is to prove that whp, in the coupling of Theorem~\ref{theorem:riordancoupling_extension}, $G \sim H_s(n,p_+)$ does not have any extra cliques that are incident to low-degree vertices of $H\sim H_r(n,\pi_+)$. In other words, we aim to show the following analogue of Lemma~\ref{lemma:lowdegreenoextra}:

\begin{lemma}\label{lemma:lowdegreenoextra_extension}
	Couple $G \sim H_s(n, p_+)$ and $H\sim H_r(n, \pi_+)$ via the coupling described in \S\ref{section:riordanalgo}. We call the hyperedges in $\mathrm{cl}(G) \setminus H$ \emph{extra cliques}. Then whp, no low-degree vertex of $H$ is incident with any extra clique in $G$.
	\end{lemma}

Observe that Lemma~\ref{lemma:extrabound} transfers to the current setting without modification. Also, the analogon of Lemma~\ref{lemma:extraavoidable} for $s$-uniform hypergraphs was proven in \cite[Le.~13]{riordan2022random}: Let $H$ be an $r$-uniform hypergraph for $r \ge 4$ and $G$ the $s$-uniform hypergraph obtained from $H$ by replacing each hyperedge by a copy of $K_r^{(s)}$ (merging multiple edges if necessary). If $G$ contains a copy of $K_r^{(s)}$ on a set of vertices that is not a hyperedge in $H$, then $H$ contains an avoidable configuration.

\begin{proof}[Proof of Lemma~\ref{lemma:lowdegreenoextra_extension}]
Consider the coupled $G$ and $H$ with $p=p_+$ and $\pi=\pi_+$ and condition on $H=H_0$ for $H_0 \notin \mc{B}$. Then in particular, there are at most $(\log n)^{8g(n)}=n^{o(1)}$ low-degree vertices in $H_0$. Let $v$ be a low-degree vertex. As in the case $s=2$, by Lemma \ref{lemma:extrabound}, the expected number of `extra' cliques $v$ gets in $G$ is at most
	\begin{align}
	\sum_{\substack{j:v \in h_j\\ h_j \notin H_0}}  \frac{\pi_j^*-\pi_+}{1-\pi_+} \sim \sum_{\substack{j:v \in h_j\\ h_j \notin H_0}}(\pi_j^*-\pi_+) = \sum_{\substack{j:v \in h_j\\ h_j \notin H_0}} (\pi_j^*-p_+^{r \choose s}) +\sum_{\substack{j:v \in h_j\\ h_j \notin H_0}}(p_+^{r \choose s} - \pi_+). \label{eq:sums_extension}
	\end{align}
As $\pi_+=p_+^{r \choose s}(1-n^{-\delta})$, the second sum can be bounded by $n^{-c}$ for some small constant $c>0$, and again, it only remains to bound the first sum in \eqref{eq:sums_extension}.
In the subsequent Lemma~\ref{lemma:sumbound_extension}, which corresponds to Lemma~\ref{lemma:sumbound}, we show that the first sum is also small. It then follows from Lemma~\ref{lemma:sumbound_extension} and our previous observations that the expected number of extra cliques containing $v$ can be bounded by $n^{-c}$ for some small constant $c>0$. As there are only $n^{o(1)}$ low-degree vertices in $H=H_0$, whp no such vertices exist, which concludes the proof of Lemma~\ref{lemma:lowdegreenoextra_extension}.   
\end{proof}

	\begin{lemma}\label{lemma:sumbound_extension}
		Suppose that $H_0 \notin \mc{B}_1 \cup \mc{B}_2 
  $, and define $\pi_j^*$ as in Lemma~\ref{lemma:extrabound}. Let $k = \mathrm{min}\big((r-1)/\binom{r}{s}, s-1-s/r\big) >0$. Let $v$ be an arbitrary vertex. Then	
  \begin{equation}\label{eq:sumbound_extension}
\sum_{\substack{j:v \in h_j\\ h_j \notin H_0}}(\pi_j^* - p_+^{r \choose s}) \le n^{-k+o(1)}.
		\end{equation}
	\end{lemma}

\begin{proof}[Proof of Lemma~\ref{lemma:sumbound_extension}]
	We mimic the proof of Lemma~\ref{lemma:sumbound}, shortening it where appropriate.	
	Let $G_0$ be the $s$-uniform hypergraph obtained by replacing every hyperedge in $H_0$ with a $K_r^{(s)}$. 

	As $H_0 \notin \mc{B}_1$, the maximum degree of $G_0$ is $n^{o(1)}$. Fixing some vertex $v$, we want to count all possible hyperedges $h_j$ which appear in the sum in \eqref{eq:sumbound_extension}, and bound their contribution to \eqref{eq:sumbound_extension}. Let $C_{h_j}$ be the clique corresponding to the hyperedge $h_j$. For some given $h_j$ so that $v \in h_j$, let the components of the subgraph induced by $h_j$ in $G_0$ have $c_1, \dots, c_t$ vertices respectively, so that $\sum_{i=1}^t c_i = r$. 
 
 Again, hyperedges $h_j$ with $c_1 = c_2 = \dots = c_t = 1$ contribute $0$ to the sum in \eqref{eq:sumbound_extension}, and in the following we only need to consider those $h_j$ where $c_i \ge s$ for at least one $i$.
		
		As before, given a fixed sequence $c_1, \dots, c_t$, there are at most $n^{t-1+o(1)}$ 
		ways to pick the clique $C_{h_j}$.	
		If $c_1=r$, then there are only $n^{o(1)}$ choices for $C_{h_j}$, and 
		\[
		\pi^*_j \le p_+ = n^{-(r-1)/\binom{r}{s}+o(1)}.\]
		This implies that the overall contribution to \eqref{eq:sumbound_extension} from $h_j$ with $c_1=r$ is at most $n^{-(r-1)/\binom{r}{s}+o(1)} \le n^{-k+o(1)}$.
		
		So in the following we only consider sequences $c_1, \dots, c_t$ where $c_i < r$ for all $i$. After conditioning on the $s$-edges $E_0$ being present, the potential clique $C_{h_j}$ is missing (at least) all the $s$-edges between different components, so
		\[
		\pi^*_j \le p_+^{{r \choose s} - \sum_{i=1}^t {c_i \choose s}}.
		\]
		Overall, given a fixed sequence $c_1, \dots, c_t$ with $c_i\leq r-t+1$ for all $i$ and $c_i \ge s$ for at least one $i$, the contribution to \eqref{eq:sumbound_extension} from the hyperedges $h_j$ corresponding to this sequence is at most
		\begin{align*}
			n^{t-1+o(1)}p_+^{{r \choose s} - \sum_{i=1}^t {c_i \choose s}} &\stackrel{\text{Lemma}~\ref{lemma:rid_of_comp}}{\leq} n^{t-1+o(1)}p_+^{{r \choose s} - \binom{r-t+1}{s}} = n^{o(1) + t - r + (r-1)\binom{r-t+1}{s}/\binom{r}{s}} \\
   & \stackrel{\text{Lemma}~\ref{lem_hstar}}{\leq}n^{o(1) +1-s+\frac{s}{r} }.
   \end{align*}
   In the above, we have used two analytical lemmata that will be proven in the appendix. 
		Since there are only finitely many potential sequences  $c_1, \dots, c_t$, this concludes the proof.
	\end{proof}

\subsection{Process coupling}
To prove Theorem~\ref{theorem:finalcouplinghyper}, it is sufficient to show the following: 

\begin{proposition}\label{prop:coupling2_extension}
We may construct a coupling $(G, \sigma_G, H, \sigma_H)$ so that all of the following hold.
	\begin{enumerate}[label={\alph*)}]
		\item $G \sim H_s(n,p_+)$, and $\sigma_G$ is a uniformly random order of the $s$-edges of $G$.
		\item $H \sim H_r(n,\pi_+)$, and $\sigma_H$ is a uniformly random order of the hyperedges of $H$.
		\item \label{item:prope_extension} Let $e_{\sigma_G(1)}, \dots, e_{\sigma_G(|G|)}$ be the $s$-edge order given by $\sigma_G$ and $h_{\sigma_H(1)}, \dots, h_{\sigma_H(|H|)}$ the hyperedge order given by $\sigma_H$. For $t \in \mathbb{N}_0$, set
		\[
		G_t^* = \{e_{\sigma_G(1)}, \dots, e_{\sigma_G(t)}\} \text{ and } H_t^* = \{h_{\sigma_H(1)}, \dots, h_{\sigma_H(t)}\} 
		\]
and define
\begin{align*}
	\HitGS &= \min\{t: \text{ every vertex of $G^*_t$ is in an $r$-clique}\},\\
\HitHS  &= \min\{t: \text{$H^*_t$ has minimum degree $1$}\}.
	\end{align*}
Then whp $\HitGS, \HitHS$ are finite. Furthermore, whp we have
		\begin{equation} \label{eq:e1_extension}
H^*_{\HitHS} \subset	\cl(G^*_{\HitGS}).
		\end{equation}
			\end{enumerate}
\end{proposition}

The remainder of this section explains the necessary steps in the proof of Proposition~\ref{prop:coupling2_extension}.

The process coupling proceeds as in Section \ref{subsection:construction}: Given the coupling $G \sim H_s(n, p_+)$ and $H \sim H_r(n,\pi_+)$ from Theorem~\ref{theorem:riordancoupling_extension}, we aim to equip $G$ and $H$ with appropriately coupled uniform orders $\sigma_G, \sigma_H$ of their respective hyperedges. Only this time, since $s\ge 3$, we can exploit the major simplification that whp, no two hyperedges of $H$ overlap in $s \ge 3$ vertices, as this would form an avoidable configuration (see Remark~\ref{remark:vertexdisjoint}). Therefore, there is no need for dummy ($s$-)edges. As $H \notin \mathcal B$ holds whp, it again suffices to specify the construction of $\sigma_G,\sigma_H$ for this case, which we assume from now on.

In the current extension, the process coupling takes the following simple form: Let
 $\{\xi_e : e \in G\}$ 
be uniform random variables on $[0,1]$, independent of each other and everything else. 
For $e \in G$, let
\[\tau(e) = \xi_e  \in [0,1]\]
be the auxiliary time of $e$, and order the $s$-edges in $G$ by the size of  $\tau(e)$ in increasing order. This defines the order $\sigma_G$, which again is uniform by symmetry. 

For $h \in H$, let $E(h)\subset G$ be the $s$-edge set of the $K_r^{(s)}$ in $G$ on the same vertex set, and set
\[
\tau(h) = \max \{ \xi_e: e \in E(h)\}.
\]
Now order the hyperedges of $H$ by the value of the auxiliary times $\tau(h)$ in increasing order, defining the order $\sigma_H$. Since no two hyperedges of $H$ overlap in at least $s \ge 3$ vertices, any two distinct hyperedges of $H$ depend on $\binom{r}{s}$ distinct and therefore independent random variables from the family $\{\xi_e : e \in G\}$. Thus, by symmetry, $\sigma_H$ is uniform.

As in §~\ref{section:induced_clique}, $\sigma_G$ induces an order of the $K_r^{(s)}$-edge sets $\{E(h): h \in H\}$, such that for all $h \in H$ we have $\tau(h) = \tau(E(h))$, so that any hyperedge $h$ and its clique appear at the same auxiliary time.
Define $(G(t))_{0 \geq t \geq 1}$ and $(H(t))_{0 \geq t \geq 1}$ as well as the times $t_G, t_H$ as in (\ref{eq:defHt})-(\ref{eq:deftH}). In the present coupling, we have $t_G \leq t_H$, and since $t_H$ is finite whp, so is $t_G$. We obtain the following identity:

\begin{lemma}\label{lemma:hittingtimesmatch_extension}
	Whp, $t_G = t_H$.
\end{lemma}

\begin{proof}
The proof proceeds exactly as the proof of Lemma~\ref{lemma:hittingtimesmatch}, so we only summarise it here. Define the hypergraphs $H_-$ and $G_-$ as in the proof of Lemma~\ref{lemma:hittingtimesmatch}, and let $\mathcal I$ denote the set of isolated vertices in $H_-$. Suppose that the whp statements of Lemmata~\ref{lemma:lowdegreenoextra_extension} and \ref{lemma:isolatedlowdegree} hold, and that $H \notin \mathcal B$. 

Then $\mathcal I$ is exactly the set of vertices not contained in any $K_r^{(s)}$ in $G_-$: First, $G_-$ does not have any other vertices $v \notin \mathcal I$ not contained in cliques, since all the hyperedges of $H_-$ are present as cliques in $G_-$. But it could potentially be that one of the vertices $v \in \mathcal I$ is covered by an extra clique in $G_-$. To exclude this possibility, observe that Lemma~\ref{lemma:isolatedlowdegree} implies that all $v \in \mathcal I$ become low-degree vertices in $H$. By Lemma~\ref{lemma:lowdegreenoextra_extension}, in the coupled hypergraph $G$, none of the vertices in $\mathcal I$ are included in any extra cliques. In particular, none of the vertices $v \in \mathcal I$ are covered by any extra cliques in $G_-$. 
Finally, all cliques covering the vertices $v \in \mathcal I$ appear in the same order as the corresponding hyperedges. This proves the claim that $t_H=t_G$.
\end{proof}

\begin{proof}[Proof of Proposition~\ref{prop:coupling2_extension}]
  By Lemma~\ref{lemma:hittingtimesmatch_extension}, whp we have $t_G=t_H$ and, moreover, $G(t_G)=G^\ast_{T_G^\ast}$ as well as $H(t_H)=H^\ast_{T_H^\ast}$. Using that $\tau(h) = \tau(E(h))$ for all $h \in H(t_H)$, the claim follows.
\end{proof}

\subsection*{Open Problems}

In their breakthrough paper~\cite{johansson2008factors}, Johansson, Kahn and Vu found thresholds not only for perfect matchings and $r$-clique factors, but also for $F$-factors whenever $F$ is a fixed strictly balanced graph. Is it possible to obtain \emph{sharp} thresholds, or even a hitting time result, for the existence of an $F$-factor? Riordan's coupling result~\cite{riordan2022random} can be extended from the case $F=K_r$ to certain \emph{nice} graphs $F$ (see Definition 9 in \cite{riordan2022random}) with some convenient properties. Thus, for these $F$, a sharp threshold for $F$-factors can be obtained via this coupling from Kahn's sharp threshold result for perfect matchings \cite{kahn2019asymptotics}. Building upon Riordan's coupling, the approach from this paper may extend to these nice $F$, yielding a hitting time result --- or perhaps some additional constraints are needed to make the proof go through in this case. However, for $F$ which are strictly balanced but not nice, it is unclear how to proceed.

One might also consider thresholds and hitting times for other hypergraph properties, which could then be transferred to the graph setting by methods similar to those in this paper. The hitting time result for perfect matchings in $r$-uniform hypergraphs, Theorem \ref{theorem:kahnhitting}, is an extension of the corresponding classic result for $r=2$ in \cite{bollobas1985} that the hitting times for the existence of a vertex cover, a perfect matching and connectivity coincide whp. The hitting time results for the connectivity thresholds were extended to hypergraphs in \cite{poole2015}, thus we know that the hitting times for minimum degree 1, connectivity and the existence of a perfect matching coincide.
Another classic result in \cite{bollobas1985} states that the hitting times for minimum degree $2$, $2$-connectivity and the existence of a Hamilton cycle coincide whp.
Equivalence of the former for the extension to hypergraphs was established in \cite{poole2015}.
The notion of a Hamilton cycle does not extend canonically to hypergraphs.
For one version, loose Hamilton cycles, there was significant progress towards the threshold \cite{dudek2011,dudek2012,ferber2015,petrova2022}, but both its exact location and the hitting time version are still outstanding.

Similarly, one may consider $2$-factors, that is, $2$-regular spanning subhypergraphs, and in particular connected $2$-factors.
As opposed to loose Hamilton cycles, where all but two $2$-degree vertices incident with each hyperedge have degree $1$, all vertices in connected $2$-factors have degree $2$.
For regular uniform hypergraphs, the location of the threshold for the existence of $2$-factors has recently been established by Panagiotou and Pasch in \cite{panagiotou2023} following the weaker result in \cite{panagiotou2019}.

A problem that is closely related to the existence of perfect matchings is the existence of exact covers. Here, we ask for a selection of vertices such that each hyperedge is incident with exactly one vertex. While the threshold has been established for uniformly random regular $r$-uniform hypergraphs in \cite{moore2016}, locating the threshold for this problem in the binomial random $r$-uniform hypergraph is still open~\cite{kalapala2008}.

\subsection*{Acknowledgements}	
This project was initiated during the research workshop of Angelika Steger's group in Buchboden, August 2021. We are grateful to Oliver Riordan for a helpful discussion.

	\bibliographystyle{plain}
	\bibliography{literature}

\appendix

\section{Appendix}
\subsection*{Proof of Lemma \ref{lemma:lowdegreenotbad}}
Let $k$ be the number of vertices in $K$, and fix a vertex $v$ in $K$ with degree $d$ (in $ K$).  Denote by $Y$ the number of copies $(v',K')$ of the rooted hypergraph $(v,K)$ in $H$, with $v'$ being a low-degree vertex. 
Then we have $E[Y]\le kE[X_K]P(D\le 7g(n)-d)$, where $D$ is binomial with parameters $\binom{n-1}{r-1}-d$ and $\pi_+$. 
As $\expe[D] =\Theta(\log n)$, the Chernoff bound, Theorem~\ref{theorem:Chernoff}, gives $P(D\le 7g(n)-d)\le e ^ {-\Theta( \log n)}=n^{-\Theta(1)}$, thus $E[Y]=n^{-\Theta(1)}$ using $E[X_K]\le n^{o(1)}$. Therefore  
whp no low-degree vertex of $H$ exists in a copy of $K$ in $H$.
\qed

\subsection*{Proof of Lemma~\ref{lemma:boundforpijprime}}
For a), considering \eqref{eq:defpijprime}, we can only have $\pi_j'=0$ if $h_j$ is the last missing hyperedge in a forbidden $3$-cycle which was excluded in the event $\mc{D}_H$. But then the corresponding triangle in $G$ is the last missing triangle in a forbidden $3$-cycle excluded in the event $\mc{D}_G$, so by  \eqref{eq:defpij} we have $\pi_j=0$.

Now let
\begin{align*}
    \mc{U} &= \mc{U}_H \cap \bigcap_{i \in Y_{j-1}} B_i\\
    \mc{D} &= \bigcap_{i \in N_{j-1} \cup S_{j-1}} B_i^c.
\end{align*}
then we have
\begin{equation} \label{eq:UandD}
\pi_j' = P(B_j \mid \mc{U} \cap  \mc{D} \cap \mc{D}_H ).
\end{equation}
For b), consider the auxiliary random hypergraph $H'$ where the hyperedges $h_i$, $i\in \bY_{j-1}$, are deterministically present, as are all the hyperedges from the clean $3$-cycles in $\mc{C}_H$ (so $\mc{U}_H$ holds), and the remaining hyperedges are drawn independently with probability $\pi$. Then as $\mc{D} \cap \mc{D}_H $ is a down-set, we can apply Harris' inequality to $H'$, giving
\[
\pi'_j =  P(B_j \mid \mc{U} \cap \mc{D} \cap \mc{D}_H ) \le  P(B_j \mid \mc{U})=\pi.
\]

For c), assume $\pi_j > \pi_j'$. As $\pi'_j \le \pi$ by b), we only need to bound $\pi_j'$ from below. 
First of all note that $\pi_j> \pi_j' \ge 0$, which by part a) implies
\begin{equation}\label{eq:pijprimepositive}
\pi_j'>0.
\end{equation}
In particular $P(B_j \mid \mc{U} \cap \mc{D})=\pi>0$, and further $P(\mc D_H \mid \mc{U} \cap \mc{D})>0$ \footnote{Since $\mc D_H$ is the event that there are no further clean $3$-cycles, which is compatible with the clean $3$-cycles in $\mc U_H$, and if any subset of hyperedges in the event $\mc U_H \cap \bigcap_{i \in Y_{j-1}}B_i$ included a forbidden clean $3$-cycle, then the last index $i$ from that subset would not have been added to $Y_{j-1}$ as then $\pi_i'=0$.}. So an application of Bayes' theorem to $P(\cdot \mid \mc{U} \cap \mc{D})$ gives  
\[
\pi_j' = \frac{P(\mc{D}_H \mid B_j \cap \mc{U} \cap \mc{D}) \cdot P(B_j \mid \mc{U} \cap \mc{D})}{P(\mc{D}_H \mid \mc{U} \cap \mc{D})} = \frac{P(\mc{D}_H \mid B_j \cap\mathcal U\cap \mathcal D)}{P(\mc{D}_H \mid \mathcal U \cap\mathcal D)}\pi.
\]
Next we split up the event $\mc{D}_H$ further, distinguishing between clean $3$-cycles that contain $h_j$ and those that do not. Let $\mc{D}_{H,1}$ be the event that $H$ does not have any forbidden clean $3$-cycles which contain $h_j$, and let $\mc{D}_{H,0}$ be the event that $H$ does not have any forbidden clean $3$-cycles which do not contain $h_j$, so we have
\[\mc{D}_H = \mc{D}_{H,0} \cap \mc{D}_{H,1}.\]
Plugging this in above, we obtain
\begin{align}
\pi_j' &=  \frac{P(\mc{D}_{H,0} \cap \mc{D}_{H,1}\mid B_j \cap\mathcal U\cap \mathcal D)}{P(\mc{D}_{H,0} \cap \mc{D}_{H,1} \mid \mathcal U \cap\mathcal D)}\pi = \frac{P(\mc{D}_{H_0} \mid  B_j \cap\mathcal U\cap \mathcal D) \cdot P(\mc{D}_{H,1} \mid \mc{D}_{H,0} \cap B_j \cap\mathcal U\cap \mathcal D)}{P(\mc{D}_{H_0} \mid \mathcal U\cap \mathcal D) \cdot P(\mc{D}_{H,1} \mid \mc{D}_{H,0} \cap \mathcal U \cap\mathcal D)}\pi \nonumber \\
&= \frac{P(\mc{D}_{H,1} \mid \mc{D}_{H,0} \cap B_j \cap\mathcal U\cap \mathcal D)}{P(\mc{D}_{H,1} \mid \mc{D}_{H,0} \cap \mathcal U \cap\mathcal D)}\pi ,\label{eq:pijprimepi}
\end{align}
because $P(\mc{D}_{H_0} \mid  B_j \cap\mathcal U\cap \mathcal D) = P(\mc{D}_{H_0} \mid  \mathcal U\cap \mathcal D)$ since both events $B_j \cap\mathcal U\cap \mathcal D$ and $\mathcal U\cap \mathcal D$ only include or exclude certain hyperedges from $H$, and the forbidden clean $3$-cycles in $\mc{D}_{H_0}$ do not contain $h_j$, so their presence is independent from $B_j = \{h_j \in H\}$.

For a lower bound on $\pi_j'$, we will simply bound the denominator in \eqref{eq:pijprimepi} from above by $1$, but we need a useful lower bound on the numerator. For this, using the union bound, we write 
\begin{align}
P(\mc{D}_{H,1} \mid \mc{D}_{H,0} \cap B_j \cap\mathcal U\cap \mathcal D) &= 1-P(\mc{D}_{H,1}^c \mid \mc{D}_{H,0} \cap B_j \cap\mathcal U \cap \mathcal D) \nonumber  \\
& \ge 1-\sum_{\substack{c=\{h_j, h_i, h_k\} \\ \text{ clean $3$-cycle, } c \notin \mc{C}_H}} P(c \subseteq H \mid \mc{D}_{H,0} \cap B_j \cap\mathcal U \cap \mathcal D) \label{eq:sumoverc}
\end{align}
where the sum goes over all potential clean $3$-cycles $c=\{h_j, h_i, h_k\}$ containing $h_j$.

The conditional probability of $\{c \subseteq H\}$ is clearly $0$ if $\{i,k\}\cap(\bN_{j-1}\cup \bS_{j-1})\neq\emptyset$, so we assume that the intersection is empty. We distinguish the contributions of $c=\{h_j, h_i, h_k\}$ depending on the number $a=|\{i,j,k\}\cap \bY_{j-1}|$ of hyperedges known to be present. Clearly we cannot have $a=3$ since $j\not\in \bY_{j-1}$. But we also cannot have $a=2$: by \eqref{eq:pijprimepositive} the numerator in \eqref{eq:pijprimepi} is positive --- but it would be $0$ if there were a forbidden $3$-cycle $c \in \mc{D}_{H,1}$ where only $h_j$ is missing, as we condition on $B_j = \{h_j \in H\}$.

So we only need to count contributions to \eqref{eq:sumoverc} from $c=\{h_j, h_i, h_k\}$ with $a\in \{0,1\}$. We start with the case $a=1$, say $i\in Y_{j-1}$. Suppose that $h_j$ consists of the three vertices $v_1, v_2, v_3$, then without loss of generality we can assume that $h_j \cap h_i = \{v_1\}$ (as this only changes our bound by a factor of $3$).  As $\tilde H \notin \mc{B}_1 \cap \mc{B}_2$, we have $d_{\tilde H}(v_1) = n^{o(1)}$, so there are $n^{o(1)}$ choices for $h_i$, and with $h_j,h_i$ fixed, there are at most $n$ choices for the missing vertex to complete $h_k$, yielding $O(n^{1+o(1)})$ possible choices for $c$ with $a=1$.

On the other hand, letting $\mcle= \mc{D}_{H,0} \cap B_j \cap\mathcal U \cap \mathcal D$, we have $P(c\setle H|\mcle)=P(h_k\in H|\mcle)$. Using the auxiliary hypergraph with the hyperedges in the events $\{B_j \cap\mathcal U\}$ included deterministically and the remainder drawn independently with probability $\pi$, by Harris' inequality and independence obtain
\[P(c\setle H \mid \mathcal E) \le P(c \subseteq H \mid B_j \cap \mc U) = \pi.
\]
So the overall contribution to the sum in \eqref{eq:sumoverc} from $c$ with $a=1$ is $O(n^{1+o(1)} \pi) = n^{-1+o(1)}$.

For $a=0$, we choose the three remaining vertices in the clean $3$-cycle $c$, giving $O(n^3)$ choices for $c=\{h_j, h_i, h_k\}$ with $a=0$. In the same way as above, we obtain $P(c\setle H|\mcle)\le\pi^2$. So the overall contribution to the sum in \eqref{eq:sumoverc} from $c$ with $a=0$ is $O(n^3 \pi^2) = n^{-1+o(1)}$. Overall, we obtain from \eqref{eq:sumoverc} that
\[
P(\mc{D}_{H,1} \mid \mc{D}_{H,0} \cap B_j \cap\mathcal U\cap \mathcal D) \ge 1-n^{-1+o(1)},
\]
and plugging this into \eqref{eq:pijprimepi} and bounding the denominator by $1$ yields the required lower bound on $\pi_j'$, concluding the proof of c).

Finally, part d) follows analogously to the equivalent observation in Riordan's proof \cite{riordan2022random}, that is, the remark in the paragraph below Equation (7).  For brevity (but at the expense of not being self-contained), we only describe the required changes, so recall the proofs from \cite{riordan2022random,heckel2020random}, and in particular the relative error 
\[Q=Q_j = \sum_{i \in N_1} p^{|E_i \setminus (E_j \cup R)|} \]
from Equation (4) in \cite{riordan2022random}, and from the first equation on page 620 in \cite{heckel2020random}, so that $\pi_j \ge p^{r \choose 2} (1-Q)$. Recall that in \cite{riordan2022random}, the set $N_1$ comprises the indices of hyperedges $h_i$, $i<j$, where we decided $A_i$ does not hold, so that the edge set $E_i$ of the corresponding clique overlaps with $E_j \setminus R$ (where $R$ are the edges already known to be present at time $j$ of the algorithm). In \cite{heckel2020random} the set $N_1=N_{\mrme}\cup N_{\mrmc}$ contains indices $N_{\mrme}$ for such clique edge sets $E_i$ and indices $N_{\mrmc}$ for such overlapping edge sets $E_i$ of clean $3$-cycles that are known not to be present. In both cases, let $D_j$ be the set of indices $i \in N_1$ so that $E_i \subset E_j \cup R$.

The point in \cite{riordan2022random} was that if $\mc B_1 \cup \mc B_2$ does not hold, then for all $j$ so that $D_j =\emptyset$ we have $Q_j=o(1)$. Thus, a function $\pi \sim p^{r \choose 2}$ can be picked so that for all such $j$, $\pi_j \ge p^3 (1-Q)  \ge \pi$ --- and thus the coupling can not fail at such a step. Hence, in particular $\pi_j < \pi$ is \emph{only} possible if $\mc B_1 \cup \mc B_2$ holds, or if $D_j \neq \emptyset$. But the latter case implies $\pi_j=0$, since the edge set $E_i$ was decided earlier not to be completely present in $G$, but $E_i \subset E_j \cup R$ and $A_j$ would imply that $E_j \cup R $ is present in $G$.

In our setting, that of \cite{heckel2020random}, the same argument can be made: Suppose $\mc B_1 \cup \mc B_2$ does not hold and that $D_j = \emptyset$ --- we aim to bound $Q$ by $o(1)$. The contribution of $N_{\mrme}$ to $Q$ can be bounded $o(1)$ exactly as in \cite{riordan2022random}, so we only consider the contribution of $N_{\mrmc}$  to $Q$. This contribution is bounded by $o(1)$ in the last equation on page 620 in \cite{heckel2020random}, except if there is a clean $3$-cycle with edge set $E_i$ so that $k=0$, $e_i=0$ (for the definition of $k, e_i$ see \cite{heckel2020random}). But in the latter case it follows that $E_i \subseteq E_j \cup R$ (as noted in \cite[p.~620]{heckel2020random}), and hence $D_j \neq \emptyset$. Thus, we can choose $\pi \sim p^{r \choose 2}$ so that for all $j$ such that $D_j = \emptyset$, $\pi_j \ge p^3 (1-Q)  \ge \pi \ge \pi_j'$. Hence, we can only have $\pi_j < \pi_j'$ if $D_j \neq 0$, and in that case $\pi_j=0$ follows in exactly the same way as above. 
This completes the proof of part d).

\subsection*{Proof of Lemma~\ref{lemma:isolatedlowdegree}}
	Let  $d \ge 0$ be an integer. Our aim is to show that the expected number of vertices which are isolated in $H_-$, and with degree $d > 7g(n)$ in $H$, is $o(1)$.
	
	So let $d >0$ be an integer and $v$ a vertex, then setting $n_r={n-1 \choose r-1}$, the probability that $v$ has degree exactly $d$ in $H \sim H(n, \pi_+)$ is 
\[
{n_r\choose d} \pi_+^{d} (1-\pi_+)^{n_r - d} \lesssim \left(\frac{en_r \pi_+}{d}  \right)^{d}  \exp\left(-n_r\pi_+\right) \lesssim    \left(\frac{en_r \pi_+}{d}  \right)^{d} n^{-1}.
\] 
Given $H$, by \eqref{eq:hminuscoupling} we obtain $H_-$ by keeping every hyperedge in $H$ independently with probability exactly $\pi_-/\pi_+$ (at least as long as $H \notin \mc{B}$, otherwise we had not defined the construction). More accurately, $H_- = H_-'$ whp, where $H_-'$ is obtained by keeping every hyperedge of $H$ independently with probability $\pi_-/\pi_+$ (whether or not $H \in \mc{B}$). So, given $v$ with degree $d$ in $H$, the probability that $v$ is isolated in $H_-'$ is exactly $(1-\pi_-/\pi_+)^{d} $. So,
the probability that $v$ has degree $d$ in $H$ but is isolated in $H_-'$ is asymptotically at most
\[
 \left(\frac{en_r (\pi_+-\pi_-)}{d}  \right)^{d} n^{-1}= \left(\frac{2eg(n)}{d}\right)^{d} n^{-1}\le  \left(\frac{6g(n)}{d}  \right)^{d} n^{-1}.
\]
Now note that, as $g(n)\rightarrow \infty$, the expected number of vertices with degree $d > 7g(n)$ in $H$ which are isolated in $H'_-$ is bounded by
\[
n\sum_{d > 7g(n)}   \left(\frac{6g(n)}{d}  \right)^{d} n^{-1} = o(1).
\]
So whp no such vertex exists, and  since whp $H_-=H_-'$, the proof is complete.
\qed

\subsection*{Proof of Lemma \ref{lemma:isonotbad}}
By \eqref{eq:hminuscoupling}, it suffices to show the statement for $H_-' \sim H(n, \pi_-)$. So let $S$ be the set of pairs of partner hyperedges in $H_-'$. 

Whp all pairs of partner hyperedges in $S$ are vertex-disjoint (see Remark~\ref{remark:vertexdisjoint}). Therefore, it suffices to show that whp there is no vertex contained in a pair of partner hyperedges which is otherwise isolated. The expected number of such vertices is bounded by
\begin{align*}
& \binom{n}{r} \binom{r}{2} \binom{n-r}{r-2} \cdot (2r-2) \cdot \pi_-^2 (1-\pi_-)^{{n-1 \choose r-1}-2} = O\big(n^{2r-2}\pi_-^2e^{-\pi_- {n-1 \choose r-1}} \big ) =\\
&= O(\log^2 n\cdot e^{-(\log n -g(n))})= n^{-1+o(1)}.
\end{align*}
An application of the first-moment method concludes the proof.
\qed

\subsection*{Proof of two analytical lemmas from Section~\ref{section:hypergraphfactor}}

\begin{lemma}\label{lemma:rid_of_comp}
Let $t \in \{2, \ldots, r-s+1\}$. For any $c_1, \ldots, c_{t} \in \{1, \ldots, r-t+1\}$ with $\sum_{i=1}^{t} c_i = r$, 
\begin{align}\label{upper_parts}
\sum_{i=1}^{t} \binom{c_i}{s} \leq \binom{r-t+1}{s}.
\end{align}
\end{lemma}

\begin{proof}[Proof of \ref{lemma:rid_of_comp}]
The claim is that the sum in (\ref{upper_parts}) is maximised if one $c_i$ takes the maximal value $r-t+1$, while the remaining ones are equal to $1$. To prove this claim, w.l.o.g., consider a non-increasing sequence $(c_1, \ldots, c_{t}) \in \{1, \ldots, r-t+1\}^t$ with $\sum_{i=1}^{t} c_i = r$ that is different from $c^*:=(r-t+1,1,\dots,1)$, such that in particular $c_1 \leq r-t$. \\ Because of $c_1\leq r-t$, the fact that $c_1 \ge c_2 \ge \ldots \ge c_t$ and $\sum_{i=1}^{t} c_i = r$, this enforces $c_2 >1$. Define $c_1' := c_1+1, c_2':= c_2 -1$ and $c_i':=c_i$ for $i \ge 3$. This operation only increases the binomial sum from the lemma:
\begin{align}\label{eq_max_1}
\sum_{i=1}^{t} \binom{c_i}{s} \leq \sum_{i=1}^{t} \binom{c'_i}{s}.
\end{align}
By construction, (\ref{eq_max_1}) reduces to showing that
\begin{align}\label{eq_max_2}
 \binom{c_1}{s} + \binom{c_2}{s} \leq \binom{c_1+1}{s} + \binom{c_2-1}{s},
\end{align}
which is equivalent to
 \begin{align}\label{eq_max_3}
\binom{c_2-1}{s-1} \leq \binom{c_1}{s-1}.
\end{align}
Since $c_2 \leq c_1$, (\ref{eq_max_3}) is satisfied. Correspondingly, any non-increasing configuration $(c_1, \ldots, c_{t})\neq c^*$ can be changed into $c^*$ in a sequence of steps that does not decrease the binomial sum. This completes the proof.
\end{proof}

\begin{lemma}\label{lem_hstar}
Consider the function $w: [2, r-s+1] \to \mathbb{R}, w(t) := t - r + (r-1)\binom{r-t+1}{s}/\binom{r}{s}$. Then 
\begin{align*}
\max_{t \in [2,r-s+1]}w(t) = w(2) = 1-s+\frac{s}{r} < 0.
\end{align*}
\end{lemma}

\begin{proof}[Proof of Lemma~\ref{lem_hstar}]
We show that $w$ is strictly convex on $[2, r-s+1]$, $w(2)=1-s+s/r$ and  $w(r-s+1) = 1-s+(r-1)/\binom{r}{s} \leq w(2)$. For this, we compute 
\begin{align*}
w'(t) = 1 - \frac{r-1}{\binom{r}{s}s!} \sum_{j = 0}^{s-1} \prod_{\substack{\ell=0 \\ \ell \not= j}}^{s-1} (r-t-\ell+1)
\end{align*}
and 
\begin{align*}
w''(t) = \frac{r-1}{\binom{r}{s}s!}\sum_{j_1 = 0}^{s-1} \sum_{\substack{j_2 = 0 \\ j_2 \not= j_1}}^{s-1} \prod_{\substack{\ell=0 \\ \ell \notin \{j_1, j_2\}}}^{s-1} (r-t-\ell+1).
\end{align*}
In particular, $w''(t) >0$ on $[2, r-s+1]$ and $w$ is strictly convex. The computation of $w(2), w(r-s+1)$ is straightforward. To show that $w(2) \ge w(r-s+1)$, it is sufficient to show that $s \binom{r}{s} \ge r (r-1)$ or equivalently, that $(r-2)\cdot \ldots \cdot (r-s+1) \ge (s-1)!$. However, the left-hand side is increasing in $r$ with $r \ge s+1$. Plugging in $r=s+1$ establishes the claim.
\end{proof}

	\end{document}